\documentclass[a4paper,10pt,oneside, reqno]{amsart}

\usepackage[T1]{fontenc}
\usepackage[utf8]{inputenc}
\usepackage{version}
\usepackage{amssymb}
\usepackage{amstext}
\usepackage{latexsym}
\usepackage{amsfonts}
\usepackage{amsthm}
\usepackage{amsmath}
\usepackage{ragged2e}
\usepackage{xcolor}
\usepackage{physics}
\usepackage{fancyhdr}
\usepackage{graphicx}
\usepackage{geometry}
\usepackage{setspace}
\usepackage{mathtools}
\usepackage{resmes}
\usepackage{enumitem}
\usepackage{etoolbox}
\usepackage{nomencl}
\usepackage[colorlinks=true,
            linkcolor=red,
            citecolor=blue]{hyperref}
            
\usepackage{fancyhdr}
\newcommand{\overbar}[1]{\mkern 1.5mu\overline{\mkern-1.5mu#1\mkern-1.5mu}\mkern 1.5mu}
\newcommand{\sub}{\subseteq}
\newcommand*\interior[1]{#1^{\mathsf{o}}}

\newcommand{\pks}{P_{k_s}}
\newcommand{\pksi}{P_{k_{s_i}}}
\newcommand{\gvalue}[1]{|#1|_G}
\newcommand{\pzbk}{P_{ZB_k}}
\newcommand{\tildeqi}{\Tilde{Q}_i^\epsilon}
\newcommand {\Odd}{\Omega_{\delta_1,\delta_2}}
\newcommand {\Qdd}{Q_{\delta_1,\delta_2}}
\newcommand{\tqil}{\Tilde{Q}_i^\lambda}

\def \R {{\mathbb {R}}}

\def \O {\Omega}
\def \N {{\mathbb{N}}}
\def \H {\mathcal{H}}
\def \L {\mathcal{L}}
\def \l {\lambda}
\def \ks {k_s}
\def \remark {\textbf{Remark. }}
\def \slim { \lim_{s\to1^-}}
\def \a {\alpha}
\def \F {\mathcal{F}}
\def \d {\delta}
\def \P {\Pi}
\def \e {\epsilon}
\def \dpe {\partial _\epsilon \Pi}
\def \dpej {\partial_\epsilon \Pi_j}
\def \o {\omega}
\def \I {\mathcal{I}}

\def \Q {\mathbb{Q}}
\def \r {\rho}

\newtheorem{theorem}{Theorem}[section]
\newtheorem*{theorem*}{Theorem}
\newtheorem{lemma}[theorem]{Lemma}
\newtheorem{proposition}[theorem]{Proposition}
\newtheorem{definition}[theorem]{Definition}
\newtheorem{corollary}[theorem]{Corollary}

\title{Gamma-convergence as $s\to1^-$ of anisotropic nonlocal fractional perimeter functionals}
\author{Alberto Fanizza}
\address{Alberto Fanizza, Institute of Analysis and Scientific Computing, TU Wien, Wiedner Hauptstraße 8-10, 1040 Vienna, Austria.}
\email{alberto.fanizza@asc.tuwien.ac.at}

\subjclass[2020]{49Q20, 52A38, 28A75}
\keywords{Fractional perimeter, nonlocal perimeter, anisotropic perimeter, $\Gamma$-convergence}

\begin{document}

\maketitle

\begin{abstract}
    We investigate the asymptotic behavior in the sense of $\Gamma(L^1_{loc})$-convergence as $s\to1^-$ of anisotropic non local $s$-fractional perimeters defined with respect to general anisotropic integration kernels $k_s(\cdot)$, under the hypothesis of pointwise convergence of such kernels. In particular, we prove the $\Gamma(L^1_{loc})$-convergence as $s\to1^-$ of the rescaled anisotropic nonlocal $s$-fractional perimeters defined with respect to the kernels $k_s(\cdot)$ to a suitable anisotropic perimeter. We do so both in $\R^n$ and on a bounded domain $\O\sub\R^n$ with Lipschitz boundary.
\end{abstract}

\tableofcontents

\section{Introduction}

Nonlocal perimeter functionals were first introduced in \cite{CafSav} by L. Caffarelli, J-.M. Roquejoffre and O. Savin in the context of the analysis of motion of interfaces in multi-phase systems when also long range interactions are accounted for. In particular, motion of surfaces by mean curvature flow can be seen as a short step limit of the scheme proposed by J. Bence, B. Merriman and S. Osher (\cite{MerBenOsh}), which can be summarized in the way that follows. Let $\O_k\sub\R^n$ be an open set with smooth boundary at a time $t_k$ and denote by $S_k=\partial\O_k$. The surface $S_{k+1}=\partial\O_{k+1}$ at a time $t_{k+1}=t_k+\delta$, where $\delta>0$ is a fixed time step, is obtained by finding the solution at a time $\e>0$ of the initial value problem for the classical heat equation:
\begin{equation*}
    \begin{cases}
        u_t(\cdot,t)& = \Delta u(\cdot,t)\\
        u(\cdot,0) &= u_k(\cdot)\,,
    \end{cases} 
\end{equation*}
where $u_k=\chi_{\O_k}-\chi_{\O_k^c}$, and then taking $\O_{k+1}=\{u(\cdot,\e)>0\}$. If $\delta=\mathcal{O}(\e^2)$ it was proved by L.C. Evans in \cite{LCE93} that this scheme provides a discrete approximation of motion by mean curvature for small time steps $\delta\to0^+$.
Nonlocal effects can be introduced in this scheme by substituting the classical Laplacian in the diffusion equation with the fractional Laplacian $(-\Delta)^\sigma$, $\sigma\in(0;1)$. When $\sigma<\frac{1}{2}$ this process converges for short time steps to a nonlocal mean curvature flow (cf. \cite{CS08}), whose stationary sets, referred as nonlocal minimal surfaces, minimize the Gagliardo seminorm:
\begin{equation*}
    [\chi_{\O}]_{H^\sigma(\R^n)}=\int_{\R^n}\int_{\R^n}\frac{\abs{\chi_\O(x)-\chi_\O(y)}^2}{\abs{x-y}^{n+2\sigma}}dxdy=2\int_\O\int_{\O^c}\frac{1}{\abs{x-y}^{n+2\sigma}}dxdy.
\end{equation*}
For $s\in(0;1)$ and Lebesgue measurable sets $A,B\sub\R^n$ let us define the following \textit{locality defect}:
\begin{equation}
    \label{locdef}
    \I_s(A,B):=\int_A\int_B\frac{1}{\abs{x-y}^{n+s}}dxdy
\end{equation}
Let $\O\sub\R^n$ be either a bounded domain with Lipschitz boundary or the whole space $\R^n$ and let $E\sub\R^n$ be a Lebesgue measurable set. The \textit{nonlocal $s$-fractional perimeter} of $E$ in $\O$ is defined as:
\begin{equation}
    \label{fracperdef}
        P_s(E,\O):=P_s^1(E,\O)+P_s^2(E,\O)
\end{equation}
        Where:
\begin{equation*}
    \begin{split}
        P_s^1(E,\O):=&\I_s(E\cap\O,E^c\cap\O)\\
        P_s^2(E,\O):=&\I_s(E\cap\O,E^c\cap\O^c)+\I_s(E^c\cap\O,E\cap\O^c).
    \end{split}
\end{equation*}
\remark In the case $\O=\R^n$ the quantity $P_s^2(E,\R^n)$ vanishes (since both integral contribution are performed on an empty integration domain), leaving us with $P_s(E,\O)=P_s^1(E,\O)$.\\
It is generally said that $P^1_s(E,\O)$ measures the $s$-perimeter of $E$ inside $\O$, since the double integration happens across the portion of the boundary of $E$ which lies inside $\O$, while $P^2_s(E,\O)$ represents the contribution of the $s$-perimeter of $E$ on the boundary of $\O$, the double integration taking place, in this case, across both the boundaries of $E$ and $\O$. This type of set functional can be seen as a $(n-s)$-dimensional perimeter measure in the sense that it is homogeneous of degree $n-s$ and it verifies a suitable coarea formula (cf. \cite{Visintincoarea}). Namely, for any $\lambda>0$ it holds that:
\begin{equation*}
    P_s(\l E,\l \O)=\l^{n-s}P_s(E,\O),
\end{equation*}
and  for any function $u\in W^{s,1}(\O)$ it holds that:
\begin{equation*}
    \frac{1}{2}[u]_{W^{s,1}(\O)}=\int_\R P_s(\{u>t\},\O)\,dt,
\end{equation*}
where $[u]_{W^{s,1}(\O)}$ denotes the Gagliardo seminorm of $u$ in the fractional Sobolev space $W^{s,1}(\O)$.\\
Perimeter functionals of this type are nonlocal in the sense that their value does not only depend of the shape of $E$ in $\O$. As for the terminology \textit{locality defect}, this originates from the fact that, roughly speaking, $\I_s(\cdot,\cdot)$ quantifies how far the functional $P^1_s(E,\cdot)$ is from being additive with respect to union of disjoint sets in the second variable.  Indeed, let $\O_1,\O_2\sub\R^n$ be two disjoint bounded domains in $\R^n$. Then:
\begin{equation*}
    P^1_s(E,\O_1\cup\O_2)=P^1_s(E,\O_1)+P^1_s(E,\O_2)+\I_s(E\cap\O_1, E^c\cap\O_2)+
    \I_s(E\cap\O_2,E^c\cap\O_1).
\end{equation*}

The functionals $P_s(E,\O)$ are isotropic, meaning that for every $R\in SO(n)$ it holds that $P_s(RE,R\O)=P_s(E,\O)$. This is clearly due to the fact that the integration kernels, defined by means of the Euclidean norm in $\R^n$ are rotation invariant.  A possible way of introducing anisotropy in these functionals is to substitute the Euclidean norm in the integration kernel by an anisotropic norm (c.f. \cite{Ludwig}). Let $K\sub\R^n$ be a convex body symmetric with respect to the origin and let
\begin{equation}
    \label{anisonorm}
    \norm{v}_X:=\inf\left\{r>0:\frac{x}{r}\in K \right\}
\end{equation}
be the anisotropic norm on $\R^n$ induced by $K$. For any $A,B\sub\R^n$ let us define the  \textit{anisotropic locality defect} with respect to $K$ as:
\begin{equation}
    \label{anisoklocdef}
    \I_{s,K}(A,B):=\int_A\int_B\frac{1}{\norm{x-y}_K^{n+s}}dxdy.
\end{equation}
For $s$, $E$, $\O$ as above the \textit{anisotropic nonlocal $s$-fractional perimeter} of $E$ in $\O$ with respect to $K$ is defined in the natural way:
\begin{equation}
    \label{fracperdef}
        P_{s,K}(E,\O):=P_{s,K}^1(E,\O)+P_{s,K}^2(E,\O)
\end{equation}
        Where:
\begin{equation*}
    \begin{split}
        P_{s,K}^1(E,\O):=&\I_{s,K}(E\cap\O,E^c\cap\O)\\
        P_{s,K}^2(E,\O):=&\I_{s,K}(E\cap\O,E^c\cap\O^c)+\I_{s,K}(E^c\cap\O,E\cap\O^c).
    \end{split}
\end{equation*}
The functionals $P_{s,K}(\cdot,\O)$ represent a nonlocal $s$-fractional counterpart of the classical anisotropic perimeter. If $\O$ and $K$ are as above and $E$ is a set of finite perimeter in the sense of Caccioppoli, the \textit{anisotropic perimeter} of $E$ in $\O$ with respect to $K$ is defined as:
\begin{equation}
    P_K(E,\O):=\int_{\partial^*E\cap\O}\norm{\nu_E}_{K^*}d\H^{n-1},
\end{equation}
where $\partial^*E$ denotes the the reduced boundary of $E$, $\nu_E$ is the measure theoretic outward pointing unit normal vector to $\partial^*E$ and $K^*$ is the polar body of $K$:
\begin{equation*}
    K^*=\left\{ v\in\R^n:x\cdot v\leq 1 \textit{ for all }\, x\in K \right\}.
\end{equation*}

\subsection{Previous contributions}
Many different properties of nonlocal fractional perimeters where investigated by a variety of authors since their introduction. As mentioned above, in \cite{CafSav} local minimizers of the functionals $P_s(\cdot,\O)$ are studied. A Lebesgue measurable set $E$ is said to be a local minimizer for $P_s(\cdot,\O)$ (or equivalently an $s$-minimal surface in $\O$) if for every Lebesgue measurable set $F\sub\R^n$ such that $E\Delta\F\Subset\O$ it holds that:
\begin{equation*}
    P_s(E,\O)\leq P_s(F,\O).
\end{equation*}
In \cite{CafSav} it was proved, among other results, that the Euler-Lagrange equation in the viscosity sense satisfied by a local minimizer $E\sub\R^n$ of $P_s(\cdot,\O)$ is given by $\Delta^{s/2}(\chi_E-\chi_{E^c})=0$ on $\partial E\cap\O$ and that $\partial E\cap \O$ is a $C^{1,\a}$ hypersurface, $\a<s$, up to a set of Hausdorff dimension at most $n-2$. Moreover, as a consequence of the Euler-Lagrange equation, it also follows that hyperplanes are $s$-minimal. Since the $C^{1,\a}$ regularity of $s$-minimal surfaces is far from the one that we have for classical minimal surfaces (see e.g. \cite{Giusti}), it was supposed that this regularity result might not be optimal for $s\to1^-$ and, alongside, the question arose weather the classical theory for minimal surfaces (ideally corresponding to the case $s=1$) could be recovered by letting $s\to1^-$. Asymptotic properties of $P_s(E,\O)$ as $s\to1^-$ were extensively studied from a geometric point of view by L. Caffarelli and E. Valdinoci (\cite{CafVald}, \cite{CV13}). In \cite{CafVald} it was proved that nonlocal $s$-minimal surfaces approach classical minimal surfaces as $s\to1^-$ and that, up to rescaling by a factor $(1-s)$, the nonlocal fractional $s$-perimeters converge to the classical perimeter measure.  These results were obtained by means of uniform density estimates and clean ball conditions as $s\to1^-$ proved in \cite{CafSav}. In \cite{SV13} O.Savin and E.Valdinoci were able to improve the bound on the Hausdorff dimension of the singular set from $n-2$ to $n-3$. In \cite{BFV14}, B. Barrios, A. Figalli and E. Valdinoci improved regularity outside the singular set from $C^{1,\a}$ to $C^\infty$. This result, combined with the ones in \cite{CV13}, gives that, for $s$ sufficiently close to $1$, $s$-minimal surfaces are $C^\infty$ outside a closed set of Hausdorff dimension at most $n-8$, thus recovering the classical regularity result. In \cite{Ambrosio}, L.Ambrosio, G. De Philippis and L. Martinazzi studied the problem from the point of view of $\Gamma$-convergence, proving that the rescaled functionals $(1-s)P_s(\cdot,\O)$, in the case of $\O\sub\R^n$ bounded domain with Lipschitz boundary, $\Gamma$-converge as $s\to1^-$ with respect to the $L^1_{loc}$ topology of $\R^n$ to $\omega_{n-1}P(\cdot,\O)$, where $P(\cdot,\O)$ denotes the classical De Giorgi's perimeter in $\O$ and $\o_{n-1}$ is the volume of the unit ball in $\R^{n-1}$. Thanks to this result, they were able to prove convergence of local minimizers of $P_s(\cdot,\O)$ to local minimizers of the classical perimeter without relying on the density estimates provided in \cite{CafSav}. For completeness, we mention that also the asymptotic properties as $s\to0+$ of $P_s(\cdot,\O)$ were studied in detail (c.f. \cite{DFPV13}).\\
As for the anisotropic case, convergence properties of the functionals $P_{s,K}(\cdot,\R^n)$ both in the case $s\to0^+$ and $s\to1^-$ were investigated by M. Ludwig in \cite{Ludwig}. In particular, it is shown there that for any Borel set $E\sub\R^n$ it holds that:
\begin{equation*}
    \lim_{s\to1^-}(1-s)P_{s,K}(E,\R^n)=P_{ZK}(E,\R^n),
\end{equation*}
where $ZK\sub\R^n$ is the \textit{moment body} of $K$, that's to say the convex body in $\R^n$ defined by the following norm associated to its polar $Z^*K$:
\begin{equation*}
    \norm{v}_{Z^*K}:=\frac{n+1}{2}\int_K\abs{v\cdot x}dx.
\end{equation*}
This result is obtained by M.Ludwig via an integral-geometric argument based on the use of a Blaschke-Petkantschin type formula. Moreover, partially exploiting the $\Gamma$-liminf inequality proved in \cite{Ambrosio} in the isotropic case, she was also able to prove $\Gamma$-convergence with respect to the $L^1_{loc}$ topology of $\R^n$. In particular, the use of Blaschke-Petkantschin formula allows her to go back from the $n$-dimensional case to the $1$-dimensional case, in which pointwise convergence can be proved directly via explicit computations. Borrowing then the $\Gamma$-liminf inequality for the $1$-dimensional case by the general result in \cite{Ambrosio}, $\Gamma$-convergence is established.

\subsection{Setting and main results of the paper}
The aim of this paper is to investigate the $\Gamma$-convergence properties of anisotropic nonlocal $s$-fractional perimeters as $s\to1^-$ in bounded domains with Lipschitz boundary by means of combining the approach used by L. Ambrosio, G. De Philippis and G. Martinazzi in \cite{Ambrosio} with the technique used by Ludwig in \cite{Ludwig}. For $s\in(0,1)$, we consider generalized anisotropic integration kernels of the form of measurable functions $k_s(\cdot):\R^n\setminus\{0\}\to\R_+$ which fulfill the following properties:
\begin{enumerate}[label=\textit{h}.\arabic*, start=0]
    \item \label{h0} for every $z\neq0$:
    \begin{equation*}
        k_s(-z)=k_s(z)
    \end{equation*}
    \item \label{h1} for every $\lambda>0$ and $z\neq 0$:
    \begin{equation*}
        \ks(\l z)=\l^{-n-s}\ks(z).
    \end{equation*}
    \item \label{h2} there exists a constant $c\geq1$ such that for every $s\in (0,1)$ and $z\neq 0$:
    \begin{equation*}
        \frac{1}{c}\frac{1}{\abs{z}^{n+s}}\leq\ks(z)\leq c \frac{1}{\abs{z}^{n+s}}.
    \end{equation*}
\end{enumerate}
In complete analogy of the cases presented above, for $A,B\sub\R^n$ Lebesgue measurable sets we define the \textit{anisotropic $k_s$-locality defect} as:
\begin{equation}
    \I_{k_s}(A,B):=\int_A\int_B k_s(x-y)dxdy.
\end{equation}
For $E\sub\R^n$ a Lebesgue measurable set and $\O\sub\R^n$ either a bounded domain with Lipschitz boundary or the whole space $\R^n$, we define the anisotropic nonlocal perimeter of $E$ in $\O$ with respect to $k_s$ in the natural way:
\begin{equation}
    \label{fracperdef}
        \pks(E,\O):=\pks^1(E,\O)+\pks^2(E,\O)
\end{equation}
        Where:
\begin{equation*}
    \begin{split}
        \pks^1(E,\O):=&\I_{k_s}(E\cap\O,E^c\cap\O)\\
        \pks^2(E,\O):=&\I_{k_s}(E\cap\O,E^c\cap\O^c)+\I_{k_s}(E^c\cap\O,E\cap\O^c).
    \end{split}
\end{equation*}
We point out that hypotheses \ref{h0}, \ref{h1}, \ref{h2} on the kernels $k_s(\cdot)$ are the ones strictly necessary for the functionals $P_{k_s}(\cdot, \O)$ to define an anisotropic $s$-fractional perimeter measure. Indeed, \ref{h0} is necessary to ensure that, for each Lebesgue measurable set $E$, $\pks(E,\O)=\pks(E^c,\O)$, while \ref{h1} provides the functionals with the degree of homogeneity $n-s$.\\
In this paper we investigate the asymptotic behavior of the functionals $(1-s)\pks(\cdot,\O)$ as $s\to1^-$ in the case in which $\O$ is the whole space $\R^n$ or a Bounded domain with Lipschitz boundary. First we show a pointwise convergence and a $\Gamma$-convergence result with respect to the $L^1_{loc}(\R^n)$ convergence in the case $\O=\R^n$.
\begin{theorem}
    \label{pointwholespace}
     Suppose $k_s(\cdot)$ converge pointwise to a measurable function $k(\cdot):\R^n\setminus\{0\}\to\R_+$ in $\R^n\setminus\{0\}$. Then, there exists a unique origin symmetric convex body $ZB_k\sub\R^n$ such that for every $E\sub\R^n$ bounded measurable set of finite perimeter:
    \begin{equation}
        \label{pointwise}
        \slim (1-s)P_{k_s}(E,\R^n)=\pzbk(E,\R^n).
    \end{equation}
\end{theorem}
This theorem only represents a slight generalization of \cite[Theorem 4]{Ludwig} to functionals defined through more general integration kernels and can be proved using the same arguments after giving a proper characterization of the properties of the kernels $k_s(\cdot)$. The following $\Gamma$-convergence result also follows directly from Theorem \ref{pointwholespace}.
\begin{corollary}
    \label{gammawholespace}
    Let $\O\sub\R^n$ be a bounded domain with Lipschitz boundary. Suppose that the functions $k_s(\cdot)$ converge pointwise to a measurable function $k(\cdot):\R^n\setminus\{0\}\to\R_+$. Then, for every Lebesgue measurable set $E\sub\R^n$,
    \begin{equation}
        \Gamma\hbox{-}\lim_{s\to1^-}(1-s)\pks(E,\R^n)=\pzbk(E,\R^n)
    \end{equation}
    with respect to the $L^1_{loc}(\R^n)$ convergence of sets, where $ZB_k\sub\R^n$ is the same as in Theorem \ref{pointwholespace}.
\end{corollary}
The main result that we will prove here is the following $\Gamma$-convergence result for the rescaled functionals $(1-s)P_{k_s}(\cdot,\O)$ as $s\to1^-$ on bounded domains with Lipschitz boundary.
\begin{theorem}
    \label{gammaonomega}
    Let $\O\sub\R^n$ be a bounded domain with Lipschitz boundary. Suppose that the functions $k_s(\cdot)$ converge pointwise to a measurable function $k(\cdot):\R^n\setminus\{0\}\to\R_+$. Then, for every Lebesgue measurable set $E\sub\R^n$,
    \begin{equation}
        \Gamma\hbox{-}\lim_{s\to1^-}(1-s)\pks(E,\O)=\pzbk(E,\O)
    \end{equation}
    with respect to the $L^1_{loc}(\R^n)$ convergence of sets, where $ZB_k\sub\R^n$ is the same as in Theorem \ref{pointwholespace}.
\end{theorem}
This result can be seen as a generalization of \cite[Theorem 2]{Ambrosio} to the anisotropic setting. The proof is obtained by means of introducing the technique proposed by M. Ludwig in \cite{Ludwig} in the arguments used by L. Ambrosio, G. De Philipps and L. Martinazzi in \cite{Ambrosio}. Their result in the isotropic case can be recovered by computing $\pzbk(\cdot,\O)$ in the case in which the kernels $k_s(\cdot)$ are the ones induced by the euclidean norm:
\begin{equation*}
    k_s(z)=\frac{1}{\abs{z}^{n+s}}.
\end{equation*}
There are two main advantages in recovering the isotropic $\Gamma$-limit as a particular case of the anisotropic one. First, the computations necessary to determine the multiplicative constant $\o_{n-1}$ come out to be slightly simplified. Moreover, the determination of $\o_{n-1}$ made in \cite[Section 3.3]{Ambrosio} in hinged on the fact that halfspaces are local minimizers for the $s$-fractional perimeter (c.f. \cite[Corollary 5.3]{CafSav}), while here we only need this property of halfspaces in dimension $1$.\\
After proving $\Gamma$-convergence for the functionals $(1-s)\pks(\cdot,\O)$, it is possible to adapt the proof of \cite[Theorem 3]{Ambrosio} in order to obtain convergence of local minimizers also in this anisotropic setting, only minor changes in the argument being necessary.
\begin{theorem}
    \label{convlocmin}
    Let $(s_i)_{i\in\N}\sub(0;1)$ be a sequence such that $s_i\to1^-$ and suppose that for each $i\in\N$ $E_i\sub\R^n$ is a local minimizer for $\pksi(\cdot,\O)$. Suppose also that there exists a measurable set $E\sub\R^n$ for which $E_i\to E$ in $L^1_{loc}(\R^n)$. Then, for every $\O'\Subset\O$:
    \begin{equation}
        \label{minbound}
        \limsup_{i\to\infty}(1-s_i)\pksi(E_i,\O')<\infty.
    \end{equation}
    Moreover, $E$ is a local minimizer for $\pzbk(\cdot,\O)$ and for every $\O'\Subset\O$ with $\pzbk(E,\partial \O')=0$:
    \begin{equation}
        \lim_{i\to\infty}(1-s_i)\pksi(E_i,\O')=\pzbk(E,\O').
    \end{equation}
\end{theorem}

\section{Preliminaries}
We dedicate this section to fixing some notation and briefly illustrating some preliminaries on Caccioppoli sets, homogeneous spaces and Blaschke-Petkantschin type formulas.

\subsection{General notation} Let us introduce here some fixed notation that will be used throughout the rest of the paper.\\
For any $K\sub\R^n$ convex body symmetric with respect to the origin the anisotropic norm induced by $K$ on $\R^n$ will be denoted by $\norm{\cdot}_K$ instead.\\
For $x\in\R^n$ and $r>0$ let $B(x,r)$ be the open ball with center $x$ and radius $r$. When $x=0$ we will use the notation $B_r=B(0,r)$.\\
Let $H$ denote the halfspace $\{x\in\R^n:x_n\leq 0\}$ and let $Q=(-\frac{1}{2},\frac{1}{2})^n\sub\R^n$ be the standard cube in $\R^n$ centered in the origin with side length one.\\
For a function $\phi:\R^n\to\R^n$ and a vector $y\in\R^n$ we will denote by $\tau_y\phi(x)=\phi(x+y)$ the $y$-translated function.\\
The Lebesgue measure on $\R^n$ will be denoted by $\L^n$, whereas for $d\geq0$ the $d$-dimensional Hausdorff measure on $\R^n$ will be denoted by $\H^d$.\\
For a measurable set $E$ and $\O\sub\R^n$ as above, the classical perimeter of $E$ inside $\O$ will be denoted by $P(E,\O)$.\\
If $G\sub\R^n$, we will denote its internal part by $G^\circ$.\\
The Lebesgue measure of the unit ball in $\R^n$, $\L^n(B(1))$, will be denoted by $\o_n$.\\

\subsection{Caccioppoli sets}
In this subsection we give some definition and recall some basic results on sets with finite perimeter, also known as Caccioppoli sets. For an exhaustive presentation of the topic we refer to the books by L.C. Evans and R.F. Gariepy (\cite{Evansgariepy}) and F. Maggi (\cite{Maggi}). Let $E\sub\R^n$ be a Lebesgue measurable set and let $\O$ be an open set.
\begin{definition}
    We say that $E$ is a set of locally finite perimeter in $\O$ if for every compact set $K\sub\O$ the following quantity is finite:
    \begin{equation}
       sup\left\{\int_E\mathrm{div}\phi\,dx:\;\phi\in C^\infty_c(K,\R^n):\sup\abs{\phi}\leq 1\right\}<\infty
    \end{equation}
    If the bound above is independent of $K$ we say that $E$ has finite perimeter in $\O$ and we define the perimeter of $E$ is $\O$ as:
    \begin{equation}
        P(E,\O):=sup\left\{\int_E\mathrm{div}\phi\,dx:\;\phi\in C^\infty_c(\O,\R^n):\sup\abs{\phi}\leq 1\right\}.
    \end{equation}
\end{definition}
Sets with locally finite perimeter can be characterized in terms of the distributional gradient of their characteristic functions. In particular, a measurable set $E\sub\R^n$ is a set of locally finite perimeter if and only if the distributional gradient of its characteristic function is a vector valued Radon measure, as specified by the following proposition (c.f.\cite[Proposition 1.21]{Maggi}).
\begin{proposition}
    \label{deffinperf}
    Let $E\sub\R^n$ be a measurable set. Then $E$ is a set of locally finite perimeter if and only if there exists a positive Radon measure $\norm{\partial E}$ and a $\norm{\partial E}$-measurable function $\nu_E:\R^n\to\R^n$ such that $\abs{\nu_E}=1$ $\norm{\partial E}$-almost everywhere on $\R^n$ and for each $\phi\in C^1_c(\R^n,\R^n)$ it holds that:
    \begin{equation}
        \label{generaldiv}
        \int_E \mathrm{div}\phi\,dx=\int_{\R^n}\phi \cdot \nu_E\,d\norm{\partial E}.
    \end{equation}
    Moreover, $E$ has finite perimeter if and only if $\norm{\partial E}(\R^n)<\infty$.
\end{proposition}
If $E\sub\R^n$ is a set of locally finite perimeter the radon measure $\norm{\partial E}$ is called the \textit{Gauss-Green measure} or the \textit{perimeter measure} of $E$. In the one-dimensional case, a set $E$ is of locally finite perimeter if and only if it is up to a set of $\L^1$-measure zero the countable union of open intervals lying at mutually positive distance (c.f. \cite[Proposition 12.13]{Maggi}).
Sets of locally finite perimeter can be also characterized in terms of the structure of their reduced boundary.
\begin{definition}
    \label{redbound}
    Let $E\sub\R^n$ be a set with locally finite perimeter in $\R^n$. We say that $x\in\R^n$ belongs to the reduced boundary of $E$, which we denote by $\partial^*E$, if $\abs{\nu_E}(x)=1$ (where $\nu_E$ is defined in Proposition \ref{deffinperf}), $\norm{\partial E}B(x,r)>0$ for every $r>0$ and:
    \begin{equation}
        \label{reducedchar}
        \lim_{r\to0^+}\frac{1}{\norm{\partial E}B(x,r)}\int_{B(x,r)}\nu_E\,d\norm{\partial E}=\nu_E(x).
    \end{equation}
\end{definition}
For every $x\in\partial^*E$ the unit vector $\nu_E(x)$ is called the measure theoretic outer unit normal vector to $E$ at $x$. One of the main tools that can be used to study the structure of sets of locally finite perimeter is the \textit{blow up} of the reduced boundary. Let $E\sub\R^n$ be a set with locally finite perimeter, let $x\in\partial^* E$ and $r>0$. The blow up of $E$ at $x$ with scaling $r>0$ is the measurable set:
\begin{equation}
    E_{x,r}\coloneq \frac{E-x}{r}.
\end{equation}
Note that, since $E$ has locally finite perimeter, then, for each $r>0$, $E_{x,r}$ has locally finite perimeter as well.\\
For $x\in\partial^*E$ define:
\begin{equation}
    H_x\coloneq\{y\in\R^n:y\cdot \nu_E(x)\leq 0\}.
\end{equation}
The following theorem (\cite[Theorem 15.5]{Maggi}, \cite[Section 5.7.2, Theorem 1]{Evansgariepy}), states that the blow-ups at a point $x$ of the reduced boundary converge to the negative halfspace $H_x$ in $L^1_{loc}(\R^n)$ as $r\to0^+$.
\begin{theorem}
    If $E\sub\R^n$ is a set of locally finite perimeter and $x\in\partial^*E$, then:
    \begin{equation}
        E_{x,r}\to H_x
    \end{equation}
    as $r\to0^+$ in $L^1_{loc}(\R^n)$.
    Moreover, it holds that:
    \begin{equation}
        \nu_{E_{x,r}}\norm{\partial E_{x,r}}\xrightarrow{*}\nu_E \norm{\partial E}
    \end{equation}
    weakly-star in the sense of vector valued Radon measures as $r\to0^+$. 
\end{theorem}
By De Giorgi's structure theorem for sets of locally finite perimeter (\cite[Theorem 15.9]{Maggi}), if $E$ is a set of locally finite perimeter, then its reduced boundary $\partial^* E$ is countably $(n-1)$-rectifiable. Namely, $\partial^*E$ is the union, up to a set of $\H^{n-1}$-measure zero, of countably many  $C^1$-hypersurfaces $(S_j)_j$, $\nu_E$ coincides on every $S_j$ with the classical outward pointing unit normal vector and $\norm{\partial E}=\H^{n-1}\resmes\partial^*E$.   

\subsection{Blaschke-Petkantschin formulas}
The technique exploited by M. Ludwig in \cite{Ludwig} is based on the application of an integral geometric identity generally known as Blaschke-Petkantschin formula. Actually, many different version of this identity exist, so it would be more appropriate to talk about a family of Blaschke-Petkantschin-like formulas. Boradly speaking, with the name Blascke-Petkantshin formulas we refer to various formulas, commonly used in integral geometry (\cite[Section 7.2]{Schneider}), that are used to compute integrals of nonnegative measurable functions over the product of some copies of a homogeneous space by means of slicing the space itself into lower dimensional subspaces. In this subsection we briefly describe the version of the Blaschke-Petkantschin formula that we will use and the tools it requires. For a detailed presentation of the topic we refer to \cite[Section 13.2]{Schneider}. Let $A(n,1)$ denote the $1$-dimensional affine Grassmannian on $\R^n$, i.e. the space of all $1$-dimensional affine subspaces of $\R^n$ (lines). Such space can be equipped with a suitably normalized, rigid-motion invariant Haar measure, that we will denote with $dL$, which can be described as follows. Every $l\in A(n,1)$ can be written as $l=\{x+tu(l):t\in\R\}$, where $u(l)\in S^{n-1}$ and $x\in u(l)^{\perp}$. Let $h:A(n,1)\to\R_+$ be a measurable function. Then:
\begin{equation*}
    \int_{A(n,1)}h(l)dL=\frac{1}{2}\int_{S^{n-1}}\int_{u^\perp}h(\{x+tu:t\in\R\})d\mathcal{H}^{n-1}(x)d\mathcal{H}^{n-1}(u).
\end{equation*}
For a more rigorous description of invariant measures on locally compact homogeneous topological spaces see \cite[Section 13.2]{Schneider}. The next theorem illustrates the version of the Blaschke-Petkantschin formula that we will use (c.f. \cite[7.2.7]{Schneider})
\begin{lemma}[Blaschke-Petkantschin formula]
    Let $g:\R^n\times\R^n\to\R_+$ be a nonnegative Lebesgue measurable function. Then:
    \begin{equation*}
        \int_{\R^n}\int_{\R^n}g(x,y)dxdy
        =\int_{A(n,1)}\left(\int_l\int_l g(x,y)\abs{x-y}^{n-1}d\H^1(x)d\H^1(y)\right)dL.
    \end{equation*}
\end{lemma}

\section{Gamma convergence in the whole space}
This section is dedicated to the proof of the pointwise convergence and $\Gamma$-convergence as $s\to1^-$ of the functionals $(1-s)\pks(\cdot,\R^n)$ (c.f. \eqref{fracperdef}), namely Theorem \ref{pointwholespace} and Corollary \ref{gammawholespace}. In section \ref{charofkern} we provide a characterization of the integration kernels $k_s(\cdot):\R^n\setminus\{0\}\to\R_+$ which verify hypotheses \ref{h0}-\ref{h2}. In section \ref{onedimcase} we recall some results concerning the one-dimensional case. Finally, sections \ref{proofpointwholespace} and \ref{proofgammawholespace} are dedicated, respectively, to the proof of Theorem \ref{pointwholespace} and Corollary \ref{gammawholespace}.

\subsection{Characterization of the Kernels}
\label{charofkern}

For clarity, throughout all this section we will denote the standard Euclidean norm by $\norm{\cdot}$. Let us first recall that a set $G\subseteq \R^n$ is called $\textit{star convex}$ with respect to a point $x_0$ if for every $y\in G$ and every $t\in [0;1]$ we have that $tx_0+(1-t)y\in G$.\\
We introduce now a simple yet useful definition. 
\begin{definition}
    Let $G\subseteq\R^n$ be non empty and star convex set with respect to the origin. We say that $G$ is $\textit{radially open}$ if for every line $r$ passing through the origin of $\R^n$ we have that $G\cap r$ is open in $r$ equipped with the Euclidean topology induced by $\R^n$.
\end{definition}
\begin{proposition}
    Let $G\subseteq\R^n$ be non empty, bounded, star convex with respect to the origin of $\R^n$ and radially open with $0 \in \interior G$. Define for every $x\in\R^n$
    \begin{equation*}
        |x|_G:=\inf \left\{r>0:\frac{x}{r}\in G \right\}
    \end{equation*}
    Then:
    \begin{enumerate}
        \item For every $\lambda>0$ $|\lambda x|_G=\lambda|x|_G$
        \item $|x|_G=0$ if and only if $|x|=0$
        \item There exists $\tau\geq 1$ such that for every $x\in\R^n$
        \begin{equation*}
            \frac{1}{\tau}\norm{x}\leq |x|_G \leq \tau \norm{x}.
        \end{equation*}
    \end{enumerate}
    Moreover $G=\left\{ x\in\R^n:|x|_G<1\right\}$.
\end{proposition}
\begin{proof}
    To prove that $G=\left\{ x\in\R^n:|x|_G<1\right\}$, start by choosing $x\in G$. Since $G$ is radially open there exists $r<1$ such that $\frac{x}{r} \in G$, which implies, by the definition of $\gvalue{\cdot}$, that $\gvalue{x}<1$, proving one of the inclusions. Vice versa, if $\gvalue{x}<1$, since $G$ is star convex with respect to the origin, for every $r\in (\gvalue{x};+\infty)$ we have that $\frac{x}{r}\in G$. In particular, taking $r=1$, we get $x\in G$.\\
    Statements $\textit{(1)}$ and $\textit{(2)}$ follow directly from the definition of $\gvalue{\cdot}$. Let us prove $\textit{(3)}.$ Suppose $B(0,m)\sub G \sub B(0,M)$ for some $0<m<M<+\infty$. For every $x\in \R^n$ such that $\gvalue{x}=1$, $x\in\overbar{G}$. Indeed, for every $x$ such that $\gvalue{x}=1$ the sequence
    \begin{equation*}
        x_n=\frac{x}{1+\frac{1}{n}}
    \end{equation*}
    lies in $G$ and converges to $x$ as $n\to+\infty$.
    So if $\gvalue{x}=1$ we have that $x\in\overbar{G}\sub\overbar{B(0,M)}$, meaning that:
    \begin{equation*}
        \frac{\norm{x}}{M}\leq1=\gvalue{x}.
    \end{equation*}
    For every $x\in\R^n$, $x\neq 0$, there exist an unique $\lambda>0$ such that $\gvalue{\lambda x}=1$. We can then write:
    \begin{equation*}
        \frac{\norm{\lambda x}}{M}\leq1=\gvalue{\lambda x}.
    \end{equation*}
    Using the $1$-homogeneity we get:
     \begin{equation*}
        \frac{\norm{x}}{M}\leq\gvalue{x}.
    \end{equation*}
    Moreover, since $B(0,m)\sub G$, if $\gvalue{x}= 1$ then $\norm{x}\geq m$,
    which can be written as:
    \begin{equation*}
        \frac{\norm{x}}{m}\geq1.
    \end{equation*}
    By homogeneity, for every $x\in \R^n$, $x\neq 0$:
    \begin{equation*}
        \frac{\norm{x}}{m}\geq\gvalue{x}.
    \end{equation*}
    Taking now $\tau=\max\left\{\frac{1}{m},M\right\}$ we conclude.\\
\end{proof}
From now on we will call a set $G$ verifying the hypotheses of the proposition above a $\textit{quasi-ball}$ and a map $|\cdot|:\R^n\to\R_+$ fulfilling properties $\textit{(1)}$, $\textit{(2)}$ and $\textit{(3)}$ a $\textit{quasi-norm}$.\\
In the proposition above we have just showed that a quasi-ball in $\R^n$ induces a quasi-norm in a natural way. As we could expect, the converse also holds true. 
\begin{proposition}
    Let $|\cdot |:\R^n\to \R_+$ be such that:
    \begin{enumerate}
        \item For every $\lambda>0$ $|\lambda x|=\lambda|x|_G$
        \item $|x|=0$ if and only if $x=0$
        \item There exists $\tau\geq 1$ such that for every $x\in\R^n$
        \begin{equation*}
            \frac{1}{\tau}\norm{x}\leq |x| \leq \tau \norm{x}.
        \end{equation*}
    \end{enumerate}
    Then there exists a bounded, star convex with respect to the origin of $\R^n$, radially open set $G\sub \R^n$ with $0\in \interior G$ such that for every $x\in\R^n$
    \begin{equation*}
        |x|=\gvalue{x}:=\inf\left\{r>0:\frac{x}{r}\in G\right\}.
    \end{equation*}
    Moreover $G=\left\{x\in\R^n:|x|<1\right\}$.
\end{proposition}
\begin{proof}
    \begin{onehalfspace}
    Let $G:=\left\{x\in\R^n:|x|<1\right\}$. If $\norm{x}<\frac{1}{\tau}$, then $|x|\leq \tau \norm{x}<1$, so $B(0,\frac{1}{\tau})\sub G$. Moreover, if $x\in G$ then $\frac{1}{\tau}\leq|x|<1$, so $G\sub B(0,\tau)$.
    Thus, $G$ is bounded and $0\in\interior G$. \\
    The proof of the fact that $G$ is star convex w.r.t. the origin and radially open is straightforward.\\
    We now show that $|\cdot|=\gvalue{\cdot}$. For every $x\in \R^n$ there exists a unique $\lambda>0$ such that $|\lambda x|=1$. For every $r>1$ then $\abs{\frac{lx}{r}}<1$, so $\gvalue{\lambda x}\leq 1$. If otherwise $0<r<1$ then $\frac{\lambda x}{r}\notin G$, so $\abs{\lambda x}=1=\gvalue{\lambda x}$. Since both maps are positively 1-homogeneous we conclude that $\abs{x}=\gvalue{x}$.
    \end{onehalfspace}
\end{proof}
We now show that our kernels $\ks(\cdot)$ can be characterized in terms of quasi-norms.
\begin{proposition}\label{charnuc}
    Let $\ks:\R^n\setminus\{0\}\to\R_+$, $s\in(0;1)$, be measurable, origin symmetric and such that it verifies \eqref{h1} and \eqref{h2}. Then there exists an origin symmetric quasi-norm $\abs{\cdot}:\R^n\to\R_+$ such that for every $z\in\R^n$, $z\neq0$:
    \begin{equation*}
        \ks(z)=\frac{1}{\abs{z}^{n+s}}.
    \end{equation*}    
\end{proposition}
\begin{proof}
    Define:
    \begin{equation*}
        \abs{z}=
        \begin{cases}
            \ks(z)^{-\frac{1}{n+s}} & \text{if } z\neq 0\\
            0 & \text{if } z=0.
        \end{cases}
    \end{equation*}
    It is straightforward to check that this is the origin-symmetric symmetric quasi-norm requested.
\end{proof}
$\remark$ According to the previous results if $G:=\left\{z\in R^n: z\neq 0,\ks(z)>1\right\}\cup \{0\}$, then $\abs{x}=\abs{x}_G$, $G=\{\abs{x}_G<1\}$ and:
\begin{equation*}
    \ks(z)=\frac{1}{\abs{z}_G^{n+s}}.
\end{equation*}
Since $G$ need not be convex in general, $\abs{\cdot}_G$ is not necessarily a norm.

\subsection{One-dimensional Case}
\label{onedimcase}
Following the steps of \cite{Ludwig}, we start by focusing on the one dimensional case, since it will be of great use later. Let us recall the definition of $s$-fractional perimeter $P_s(E,\R)$ of a Borel set $E\sub\R$ for $s\in (0,1)$:
\begin{equation*}
    P_s(E,\R^n):= \int_E\int_{E^c} \frac{1}{\abs{x-y}^{s+1}}dxdy,
\end{equation*}
where $\abs{\cdot}$ denotes here the usual norm on $\R$.
The following proposition was first proved by Ludwig in \cite{Ludwig}.
\begin{proposition}[\cite{Ludwig}, Lemma 1]
    \label{lemma1ludw}
    Let $E\sub \R$ be a bounded measurable set of finite perimeter. Then:
    \begin{equation}
        \label{a1}
        \lim_{s\to1^-}(1-s)P_s(E,\R)=\H^0(\partial^*E)
    \end{equation}
    and for $\frac{1}{2}\leq s<1$:
    \begin{equation}
        \label{a2}
        (1-s)P_s(E,\R)\leq 8\H^0(\partial^*E)\max\{1,\mathrm{diam}(E)\}.
    \end{equation}
\end{proposition}
The next result is a $\Gamma$-liminf inequality for $(1-s)P_s(\cdot,\R)$, which, combined with \eqref{a1}, implies the $\Gamma$-convergence. This result is the one-dimensional case of \cite[Theorem 2]{Ambrosio} and can be proved avoiding the explicit computations of the geometric constant needed in the general case. We postpone the proof to section \ref{sectiondim1}.
\begin{proposition}
    \label{1dliminf}
    Let $s_i\to1^-$ and $E_i$, $E\sub \R$ be bounded measurable sets such that $E_i\to E$ in $L^1_{loc}(\R)$. Then:
    \begin{equation}
        \liminf_{i\to\infty}(1-s_i)P_{s_i}(E_i,\R)\geq P(E,\R).
    \end{equation}
\end{proposition}

\subsection{Proof of Theorem \ref{pointwholespace}}
\label{proofpointwholespace}
    From Proposition \ref{charnuc} we get that for every $s\in(0,1)$ there exists an origin symmetric quasi-norm $\abs{\cdot}_{k_s}:\R^n\to\R_+$ such that for every $z\neq 0$:
    \begin{equation*}
        k_s(z)=\frac{1}{\abs{z}_{k_s}^{n+s}}.
    \end{equation*}
    It is easy to check that the pointwise convergence implies that also $k(\cdot)$ is an origin symmetric nucleus that verifies \eqref{h1} and \eqref{h2} with $s=1$ and the same constant $c$. Let us denote with $\abs{\cdot}_k$ the associated quasi-norm. Clearly, for every $z\neq 0$:
    \begin{equation*}
        \frac{1}{\abs{z}_{k_s}^{n+s}}\longrightarrow\frac{1}{\abs{z}_{k}^{n+1}}
    \end{equation*}
    as $s\to1^-$.\\
    For every $x,y\in\R^n$, $x\neq y$, call:
    \begin{equation*}
        u(x,y)\coloneq \frac{x-y}{\norm{x-y}}\in S^{n-1}.
    \end{equation*}
    Since $\abs{\cdot}_{k_s}$ is origin symmetric and positively $1$-homogeneous, calling $l$ the affine line in $\R^n$ that contains $x-y$, we can write:
    \begin{equation*}
        \abs{x-y}_{k_s}=\norm{x-y}\abs{u(l)}_{k_s},
    \end{equation*}
    the term $\abs{u(l)}_{k_s}$ only depending on $l$. By the Blaschke-Petkantschin formula:
    \begin{equation*}
        \begin{split}
            P_{k_s}(E,\R^n) & = \int_{E}\int_{E^c}\frac{1}{\abs{x-y}_{k_s}^{n+s}}dxdy\\
            & = \int_{E\cap l\neq \emptyset}\frac{1}{\abs{u(l)}_{k_s}^{n+s}}\left(\int_{E\cap l}\int_{E^c\cap l}\frac{1}{\norm{x-y}^{s+1}}d\H^1(x)d\H^1(y)\right )dL
        \end{split}
    \end{equation*}
    where $\left\{E\cap l \neq \emptyset\right\}=\left\{l\in A(n,1):l\cap E\neq\emptyset\right\}$.
    Since for $L$-almost every $l\in A(n,1)$ such that $E\cap l\neq \emptyset$ it holds that $E\cap l$ has finite perimeter and $\partial^*(E\cap l)=\partial^*E \cap l$ (see \cite[Proposition 14.5, Theorem 18.11 and remark 18.13]{Maggi}), we have by proposition \ref{lemma1ludw}  that:
    \begin{equation*}
        \slim (1-s)\int_{E\cap l}\int_{E^c\cap l}\frac{1}{\norm{x-y}^{s+1}}d\H^1(x)d\H^1(y)=\H^0(\partial^*E\cap l).
    \end{equation*}
    Moreover, the following estimates hold true for $1/2\leq s< 1$:
    \begin{equation*}
        \begin{split}
            \int_{E\cap l}\int_{E^c\cap l}\frac{1}{\norm{x-y}^{s+1}}d\H^1(x)d\H^1(y) & \leq 8\H^0(\partial^*E\cap l)\max\left\{1,\text{diam}(E\cap l)\right\},\\
            \frac{1}{\abs{u(l)}_{k_s}^{n+s}} &\leq c\frac{1}{\norm{u(l)}^{n+s}}=c
        \end{split}
    \end{equation*}
    where $c\geq1$ is the constant in \eqref{h2}.\\
    For $u\in S^{n-1}$, let us denote with $E|u^\perp$ the projection of $E$ on $u^\perp$ and with $r(u)$ the line through the origin containing $u$. 
    By the dominated convergence theorem:
    \begin{equation*}
        \begin{split}
            \slim(1-s)P_{k_s}(E,\R^n) & =\int_{E\cap l \neq \emptyset}\frac{\H^0(\partial^*E\cap l)}{\abs{u(l)}_k^{n+1}}dL\\
            & = \frac{1}{2}\int_{S^{n-1}}\int_{E|u^\perp}\frac{\H^0(\partial^*E\cap (y+ r(u))}{\abs{u}_k^{n+1}}d\H^{n-1}(y)d\H^{n-1}(u)\\
            & = \frac{1}{2}\int_{S^{n-1}}\int_{\partial^*E}\frac{\abs{u\cdot \nu_E(x)}}{\abs{u}_k^{n+1}}d\H^{n-1}(x)d\H^{n-1}(u).
        \end{split}
    \end{equation*}
    The last step is due to \cite[Theorem 1]{Wieacker}, which is the measure theoretic equivalent of the formula for the area of the projection of smooth surfaces on planes.\\
    Let us call now $B_k=\left\{x:\abs{x}_k<1\right\}$. Since $B_k$ is bounded and star convex with respect to the origin we can compute:
    \begin{equation*}
        \begin{split}
        \frac{1}{2}\int_{S^{n-1}}\frac{\abs{u\cdot \nu_E(x)}}{\abs{u}_k^{n+1}}d\H^{n-1}(u) & =   \frac{n+1}{2}\int_{S^{n-1}}\frac{\abs{u\cdot \nu_E(x)}}{\abs{u}_k^{n+1}}\frac{1}{n+1}d\H^{n-1}(u)\\
        & = \frac{n+1}{2}\int_{S^{n-1}}\left(\int_0^{1/\abs{u}_k}\abs{u \cdot \nu_E(x)}t^{n}dt\right )d\H^{n-1}(u)\\
        & = \frac{n+1}{2}\int_{B_k}\abs{y\cdot\nu_E(x)}dy
        \end{split}
    \end{equation*}
    Thus we obtain:
    \begin{equation*}
        \slim(1-s)P_{k_s}(E,\R^n) = \int_{\partial^*E}\frac{n+1}{2}\int_{B_k}\abs{y\cdot\nu_E(x)}dy\;d\H^{n-1}(x).
    \end{equation*}
    We define now for $v\in\R^n$:
    \begin{equation*}
        \label{starnorm}
        \abs{v}_{Z^*B_k}:= \frac{n+1}{2}\int_{B_k}\abs{x\cdot v}dx.
    \end{equation*}
    $\abs{\cdot}_{Z^*B_k}$ is a norm on $\R^n$. Let us prove that if $\abs{v}_{Z^*B_k}=0$ for some $v\in R^n$ then $v=0$. All the other properties of a norm follow directly from the definition. Suppose $v\neq 0$. Since $0\in\interior{B_k}$ there exists $\lambda>0$ such that $\lambda v\in \interior{B_k}$. Let $\delta>0$ be such that $B(\lambda v,\delta)\sub B_k$ and for all $x\in B(\lambda v,\delta)$ $\abs{\lambda v \cdot x}>0$. Then:
    \begin{equation*}
         \lambda \abs{v}_{Z^*B_k}=\abs{\lambda v}_{Z^*B_k} =  \frac{n+1}{2}\int_{B_k}\abs{x\cdot \lambda v}dx \geq\frac{n+1}{2}\int_{B(\lambda v,\delta)}\abs{x\cdot \lambda v}dx>0,
    \end{equation*}
    which gives a contradiction. Since $\abs{\cdot}_{Z^*B_k}$ is a norm on $\R^n$ the set $Z^*B_k:=\left\{v\in\R^n: \abs{v}_{Z^*B_k}\leq 1\right\}$ is a convex body in $\R^n$ symmetric with respect to the origin. Let us define $ZB_k:=(Z^*B_k)^*$. Then $ZB_k$ is a convex body symmetric with respect to the origin, its polar body is $(ZB_k)^*=(Z^*B_k)^{**}=Z^*B_k$ and:
    \begin{equation*}
        \slim(1-s)P_{k_s}(E,\R^n) = \int_{\partial^*E}\abs{\nu_E(x)}_{Z^*B_k}d\H^{n-1}(x)=\pzbk(E,\R^n).
    \end{equation*}        
\remark The moment body is defined in general for a $\textit{star body}$, that's to say compact set, which is star convex with respect to the origin and such that its radial function is continuous (see \cite[Sections 1.7, 10.8]{Schneiderconvex} for an exhaustive description of star and moment bodies). Since in our case the set $B_k$ is not a star body in general, $ZB_k$ must be intended here only as a suitable generalization of such a concept, for which we maintain the same symbol used in the literature. However, in the case of a kernel induced by a norm $\norm{\cdot}_K$ on $\R^n$, $ZB_k$ reduces to $ZK$, the moment body of $K$.

\subsection{Proof of Corollary \ref{gammawholespace}}
\label{proofgammawholespace}
    By the Blaschke-Petkantschin formula and the result for the $1$-dimensional case (Proposition \ref{1dliminf}):
    \begin{equation*}
        \begin{split}
            \liminf_{i\to\infty}(1-s_i)P_{k_{s_i}}(E_i,\R^n) & = \liminf_{i\to\infty}\int_{E\cap l \neq \emptyset}\frac{(1-s_i)P_{k_{s_i}}(E_i\cap l,\R)}{\abs{u(l)}_{k_{s_i}}^{n+s_i}}dL\\
            & \geq \int_{E\cap l\neq \emptyset}\frac{\H^0(\partial^*E\cap l)}{\abs{u(l)}_k^{n+1}}dL\\
            & =\pzbk(E,\R^n).\\
        \end{split}
    \end{equation*}

\section{Gamma-convergence on bounded Lipschitz domains}

This section is dedicated to the analysys of the $\Gamma(L^1_{loc})$-convergence properties of the anisotropic nonlocal $k_s$-perimeters $(1-s)\pks(\cdot,\O)$ as $s\to1^-$ where $
\O\sub\R^n$ is a bounded domain with Lipschitz boundary. The main goal of this section is prove Theorem \ref{gammaonomega}, namely:
\begin{itemize}
    \item Whenever $s_i\to 1^-$ and $E_i,E\sub\R^n$ are measurable sets such that $\chi_{E_i}\to\chi_E$ in $L^1_{loc}(\R^n)$:
    \begin{equation}
        \label{liminf}
        \liminf_{i\to\infty}(1-s_i)\pksi(E_i,\O)\geq P_{ZB_k}(E,\O)
    \end{equation}
    \item For any measurable set $E$ and sequence $s_i\to 1^-$ there exists a sequence of measurable sets $E_i$ such that $\chi_{E_i}\to\chi_E$ in $L^1_{loc}(\R^n)$ and:
    \begin{equation}
        \label{limsup}
        \limsup_{i\to\infty}(1-s_i)\pksi(E_i,\O)\leq P_{ZB_k}(E,\O),
    \end{equation}
\end{itemize}
In section \ref{equicoerc} we state and prove an equicoercivity theorem for the functionals $(1-s)\pks(\cdot,\O)$. This result is proved in the isotropic case in \cite[Theorem 1]{Ambrosio}; for completeness, we reproduce here the proof in our anisotropic setting. Section \ref{isotropiclemmassect} is dedicated to the proof of some useful lemmas regarding the isotropic case. In sections \ref{proofgammaliminfbounded} and \ref{proofgammalimsupbounded} we prove respectively the $\Gamma$-liminf inequality \ref{gammaliminf} and the $\Gamma$-limsup inequality \ref{gammalimsup}. In section \ref{sectiondim1} we provide a simpler proof of the $\Gamma$-liminf inequality in dimension $1$ (c.f. \ref{1dliminf}).

\subsection{Equi-coercivity}
\label{equicoercivity}
This section is dedicated to showing that the equicoercivity result proved in \cite[Theorem 1]{Ambrosio} also works in our anisotropic setting.
\begin{theorem}
    \label{equicoerc}
    Let $s_i\to1^-$ and assume that $E_i\sub\R^n$ are Lebesgue measurable sets such that, for every $\O'\Subset \O$, $\O'$ open, it holds that:
    \begin{equation*}
        \sup_i(1-s_i)\pksi(E_i,\O')<+\infty.
    \end{equation*}
    Then $(E_i)_i$ is relatively compact in $L^1_{loc}(\O)$and any limit point has locally finite perimeter in $\O$.
\end{theorem}
The proof is a direct application of the Riesz-Fréchet-Kolmogorov compactness theorem, and relies on some  preliminary estimates (\cite[Proposition 4, Proposition 5, Proposition 6]{Ambrosio}) that we recall here without proof.
\begin{proposition}[{Hardy's Inequality, c.f. \cite[Proposition 4]{Ambrosio}}]
    Let $g:\R\to\R_+$ be a Borel measurable function. Then, for every $s>0$, $r>0$:
    \begin{equation}
        \label{hardy}
        \int_0^r\frac{1}{\xi^{n+s+1}}\int_0^\xi g(t) dt d\xi \leq\frac{1}{n+s}\int_0^r\frac{g(t)}{t^{n+s}}dt.
    \end{equation}    
\end{proposition}
\begin{proposition}[{\cite[Proposition 5]{Ambrosio}}]
    Let $u \in L^1(\O)$, $\O'\Subset \O$ open, $h\in\R^n$ such that $\abs{h}<\text{dist}(\O',\O)/2$. Then, for every $z\in(0,\abs{h}]$ it holds that:
    \begin{equation}
        \label{transest}
        \norm{\tau_h u-u}_{L^1(\O')}\leq C(n)\frac{\abs{h}}{z^{n+1}}\int_{B_z}\norm{\tau_\xi u -u}_{L^1(J_\abs{h}(\O'))}d\xi,
    \end{equation}
    where $J_\abs{h}(\O')\coloneq\{x\in\O:\mathrm{dist}(x,\O')<\abs{h}\}$.
\end{proposition}
Using these two results one can prove the following proposition, which is used in the proof of the equi-coercivity theorem.
\begin{proposition}[{\cite[Proposition 6]{Ambrosio}}]
    Let $u\in L^1(\O)$, $\O'\Subset\O$ and $s\in(0,1)$. Then, for every $h\in\R^n$ such that $0<\abs{h}<\text{dist}(\O',\partial \O)/2$, it holds that:
    \begin{equation}
        \label{transest2}
        \frac{\norm{\tau_h u-u}_{L^1(\O')}}{\abs{h}^s}\leq C(n)(1-s)\int_{B(0,\abs{h})}\frac{\norm{\tau_\xi u -u}_{L^1(J_\abs{h}(\O'))}}{\abs{\xi}^{n+s}}d\xi.
    \end{equation}
\end{proposition}
For readability, from now on for $u\in\L^1(\O)$ and $s\in(0;1)$ we will denote the Gagliardo seminorm of $u$ in $W^{s,1}(\O)$ by:
\begin{equation}
    \F_s(u,\O):=\int_\O\int_\O\frac{\abs{u(x)-u(y)}}{\abs{x-y}^{n+s}}dxdy.
\end{equation}
\begin{proof}[Proof of Theorem \ref{equicoerc}]
    Let us call for simplicity $u_i=\chi_{E_i}$ and $\mathcal{G}=\{u_i\}_i$. We need to prove that for every compact $K\Subset \O$ the family $\mathcal{G}|_K=\{u_i|_K\}_i$ is relatively compact in $L^1(K)$. In order to do so, let $\O'$ be an open set such that $K\Subset \O' \Subset \O$ and call $\delta=\text{dist}(K,\partial \O')$. Fix $\epsilon>0$ and choose $0<\eta<\delta/2$ such that for any $h\in\R^n$ with $\abs{h}<\eta$:
    \begin{equation*}
        \L^n(K\Delta (K+h))<\epsilon,
    \end{equation*}
    where $K+h=\left\{x+h:x\in K\right\}$. A simple computation shows that if $\abs{h}<\eta$:
    \begin{equation*}
        \norm{\tau_h (u_i|_K)-u_i|_K}_{L^1(\R^n)} \leq \norm{\tau_h u_i-u_i}_{L^1(K)}+\epsilon.
    \end{equation*}
    Applying now \eqref{transest2} we get that:
    \begin{equation*}
        \norm{\tau_h u_i-u_i}_{L^1(K)}\leq C(n)\abs{h}^{s_i}(1-s_i)\F_{s_i}(u_i,\O').
    \end{equation*}
    Using hypotesis \eqref{h2}, since $u_i=\chi_{E_i}$ we can estimate:
    \begin{equation*}
        \F_{s_i}(u_i,\O')\leq 2c\, \pksi^1(E_i,\O').
    \end{equation*}
    Thus:
    \begin{equation}
        \label{tbound}
        \norm{\tau_h u_i-u_i}_{L^1(K)}\leq C(n)\abs{h}^{s_i}(1-s_i)\pksi^1(E_i,\O'),
    \end{equation}
    and:
    \begin{equation}
        \norm{\tau_h (u_i|_K)-u_i|_K}_{L^1(\R^n)}\leq C(n)\abs{h}^{s_i}(1-s_i)\pksi^1(E_i,\O')+\epsilon
    \end{equation}
    By hypothesis, $\sup_i(1-s_i)\pksi^1(E_i,\O')<+\infty$, so we have that:
    \begin{equation*}
        \lim_{\abs{h}\to 0}\sup_i \norm{\tau_h (u_i|_K)-u_i|_K}_{L^1(\R^n)} \leq\epsilon.
    \end{equation*}
    Since $\epsilon>0$ is arbitrary we conclude that:
    \begin{equation*}
        \lim_{\abs{h}\to 0}\sup_i \norm{\tau_h (u_i|_K)-u_i|_K}_{L^1(\R^n)}=0.
    \end{equation*}
    By Riesz-Fréchet-Kolmogorov compactness criterion (\cite[Theorem 4.26]{Brezis}), $\mathcal{G|_K}$ is relatively compact in $L^1(K)$.\\
    Let now $E$ be a limit point in $L^1_{loc}(\R^n)$ of a subsequence of sets $E_{i_j}$. By \eqref{tbound}:
    \begin{equation*}
         \norm{\tau_h u_{i_j}-u_{i_j}}_{L^1(K)}\leq C(n,K)\abs{h}^{s_{i_j}}.
    \end{equation*}
    Passing to the limit with respect to $j$ we obtain for $\abs{h}$ small enough:
    \begin{equation*}
        \norm{\tau_h\chi_E-\chi_E}_{L^1(K)}\leq C(n,K)\abs{h},
    \end{equation*}
    which implies that $\chi_E\in BV_{loc}(\O)$ (\cite[Proposition 9.3]{Brezis}).    
\end{proof}

\subsection{Lemmas on the isotropic case}
\label{isotropiclemmassect}

We anticipate here two lemmas on the isotropic case which we will make great use of in the subsequent proof of the $\Gamma$-convergence. These lemmas are, in fact, just particular cases of \cite[lemma 8]{Ambrosio}. Since in such cases the proofs come out slightly simplified, we provide them here.\\
For any $\d\geq0$ and $\O$ as above let:
\begin{equation*}
    \O_\d\coloneq\{x\in\O:\mathrm{dist}(x,\O^c)<\d\},
\end{equation*}
and:
\begin{equation*}
    \O_\d^c\coloneq\{x\in\O^c:\mathrm{dist}(x,\O)<\d\}.
\end{equation*}
Note that if $\d=0$ then both these sets are empty.\\
Moreover, for $\d_1,\d_2\geq 0$ define:
\begin{equation*}
 \Odd\coloneq \O_{\d_1}\cup\O_{\d_2}^c.
\end{equation*}
Let $Q$ denote the standard open unit cube in $\R^n$:
\begin{equation}
    Q=\left(-\frac{1}{2},\frac{1}{2}\right)^n.
\end{equation}
Coherently with the notation proposed above, for any $\d\geq0$ we will call $Q_\delta\coloneq\{x\in Q:\mathrm{dist}(x,Q^c)<\delta\}$.\\
\begin{lemma}
    \label{lemma81}
   There exists a positive constant $C(n)<\infty$ such that for any $\d_1,\d_2\in \Q$ with $0\leq\d_1,\d_2<\frac{1}{2}$:
    \begin{equation}
        \limsup_{s\to1^-}(1-s)P^1_s(H,\Qdd)\leq C(n)P(H,\Qdd).
    \end{equation}
    
\end{lemma}
\begin{proof}
    If both $\d_1$ and $\d_2$ are zero the statement is trivial, so we can assume otherwise. Let us fix $\l>0$ such that we may find a finite number of open cubes $Q_i^\l\sub \Qdd$, $i=1,...,N_\l$ with the following properties:
    \begin{enumerate}[start=1, label=\roman*.]
        \item Every $Q_i^\l$ has side length $2\l$, its barycenter lies on $H\cap\Qdd$ and its sides are either orthogonal or parallel to $H$.
        \item Setting
        \begin{equation*}
            P_\l\coloneq \bigcup_{i=1}^{N_\l}Q_i^\l,
        \end{equation*}
        we have that $\H^{n-1}((H\cap\Qdd)\setminus P_\l)=0$.
    \end{enumerate}
    For $x\in H\cap\Qdd$ set:
    \begin{equation*}
        I_s(x)\coloneq \int_{H^c\cap\Qdd}\frac{1}{\abs{x-y}^{n+s}}dy.
    \end{equation*}
    Call $H^-_\l\coloneq\{x\in\R^n:-\l<x_n\leq 0\}$.\\
    \textit{Case 1.} If $x\in (H\cap\Qdd)\setminus H^-_\l$ we have that:
    \begin{equation*}
        I_s(x)\leq\int_{B(x,\l)^c}\frac{1}{\abs{x-y}^{n+s}}dy=\frac{n\o_n}{s\l^s}.
    \end{equation*}
    Hence:
    \begin{equation*}
        \int_{(H\cap\Qdd)\setminus H^-_\l} I_s(x)dx\leq\frac{n\o_n\L^n(Q)}{s\l^s},
    \end{equation*}
    which implies:
    \begin{equation*}
        \limsup_{s\to1^-}(1-s)\int_{(H\cap\Qdd)\setminus H^-_\l} I_s(x)dx=0.
    \end{equation*}
    \textit{Case 2.} Take $x\in H^-_\l\cap\Qdd$. By construction up to a set of $\L^n$ measure zero $H^-_\l\cap\Qdd=P_\l$.\\
    For every $\e>0$ and $i=1,...,N_\l$ call $\tqil$ the dilation of the cube $Q_i^\l$ by a factor $(1+\e)$.\\
    Let us write:
    \begin{equation*}
        \begin{split}
            I_s(x)&=\int_{(H^c\cap\Qdd)\cap B(x,\e\l)^c}\frac{1}{\abs{x-y}^{n+s}}dy+\int_{(H^c\cap\Qdd)\cap B(x,\e\l)}\frac{1}{\abs{x-y}^{n+s}}dy\\
            &=I^1_s(x)+I^2_s(x).
        \end{split}
    \end{equation*}
    Arguing as above we get that:
    \begin{equation*}
        \int_{H^-_\l\cap\Qdd}I^1_s(x)dx\leq\frac{n\o_n\L^n(Q)}{s\e^s\l^s}.
    \end{equation*}
    As for $I^2_s(x)$, observe that every $\tqil$ has side length $2\l(1+\e)$. So, if $x\in Q_i^\l$ and $y\in B(x,\l\e)$, then $y\in\tqil$. This leads to the following estimate:
    \begin{equation*}
        \begin{split}
            \int_{H^-_\l\cap\Qdd}I^2_s(x)dx&\leq\sum_{i=1}^{N_\l}\int_{H\cap\tqil}\int_{H^c\cap\tqil}\frac{1}{\abs{x-y}^{n+s}}dxdy\\
            &=N_\l P_s^1(H,\tqil)=N_\l 2^{n-s}(\l+\l\e)^{n-s}P_s^1(H,Q).
        \end{split}
    \end{equation*}
    Setting now:
    \begin{equation*}
        C(n)\coloneq\limsup_{s\to1^-}(1-s)P_s^1(H,Q),
    \end{equation*}
    by the arbitrariness of $\e>0$ we get that:
    \begin{equation}
        \limsup_{s\to1^-}(1-s)P_s^1(H,\Qdd)\leq C(n)N_\l(2\l)^{n-1}=C(n)P(H,\Qdd).
    \end{equation}
\end{proof}
\begin{lemma}
    \label{lemma82}
    It holds that:
    \begin{equation}
        \lim_{s\to1^-}(1-s)\int_{H\cap Q}\int_{H^c\cap Q^c}\frac{1}{\abs{x-y}^{n+s}}dxdy=0.
    \end{equation}
\end{lemma}
\begin{proof}
    Fix $\d\in\Q$, $\d>0$ and for $x\in H\cap Q$ let:
    \begin{equation}
        I_s(x)\coloneq\int_{H^c\cap Q^c}\frac{1}{\abs{x-y}^{n+s}}dy.
    \end{equation}
    First notice that with the same argument used in the previous proof we get that if $x\in H\cap (Q\setminus Q_\d)$:
    \begin{equation*}
        I_s(x)\leq\frac{n\o_n}{s\d^s}.
    \end{equation*}
    If instead $x\in H\cap Q_\d$ write:
    \begin{equation*}
        \begin{split}
             I_s(x)&=\int_{H^c\cap Q^c_\d}\frac{1}{\abs{x-y}^{n+s}}dy +\int_{H^c\cap(Q^c\setminus Q^c_\d)}\frac{1}{\abs{x-y}^{n+s}}dy\\
             &\leq \int_{H^c\cap Q^c_\d}\frac{1}{\abs{x-y}^{n+s}}dy +\frac{n\o_n}{s\d^s}.
        \end{split}
    \end{equation*}
    Thus we obtain:
    \begin{equation*}
        \int_{H\cap Q}\int_{H^c\cap Q^c}\frac{1}{\abs{x-y}^{n+s}}dxdy\leq \frac{2n\o_n\L^n(Q)}{s\d^s}+P_s^1(H,Q_{\d,\d}),
    \end{equation*}
    By lemma \ref{lemma81}:
    \begin{equation*}
        \limsup_{s\to1^-}(1-s) \int_{H\cap Q}\int_{H^c\cap Q^c}\frac{1}{\abs{x-y}^{n+s}}dxdy\leq C(n)P(H,Q_{\d,\d}).
    \end{equation*}
    Letting now $\d\to0^+$ we get the thesis.
\end{proof}

\subsection{$\Gamma$-liminf Inequality}
\label{proofgammaliminfbounded}

In order to prove the $\Gamma$-liminf inequality we will make use of the gluing construction proposed in \cite{Ambrosio}. Roughly speaking, given $\d_1>\d_2>0$ and two measurable sets $E_1,E_2\sub\R^n$, the gluing technique consists in finding a new measurable set $F\sub\R^n$ that agrees with $E_1$ in $\O\setminus \O_{\d1}$ and with $E_2$ in $\O_{\d_2}$ and such that the anisotropic $s$-fractional perimeter of $F$ in $\O$ is adequately controlled by $E_1$ and $E_2$. In particular, given for $s\in(0,1)$ measurable sets $E_s\to E$ as $s\to1^-$ in $L^1_{loc}(\R^n)$, this construction allows us to find measurable sets $F_s$ such that $F_s\to E$ as $s\to1^-$, every $F_s$ agrees with $E$ in $\O_\d$ for some $\d>0$ and $(1-s)\pks^1(F_s,\O)$ is asymptotically controlled by $(1-s)\pks^1(E_s,\O)$ as $s\to1^-$.\\
We start by introducing some necessary notation.
For $s\in(0,1)$, $k_s(\cdot)$ as before and a measurable function $u:\O\to\R$ set:
\begin{equation}
    \label{ksseminorm}
    \F_{k_s}(u,\O)\coloneq \int_{\O\times \O}k_s(x-y)\abs{u(x)-u(y)}dxdy.
\end{equation}
Note that since that $k_s(\cdot)$ verifies hypothesis \eqref{h2} there exists a constant $c\geq1$ such that:
\begin{equation}
    \frac{1}{c}\F_s(u,\O)\leq \F_{k_s}(u,\O)\leq c\,\F_s(u,\O).
\end{equation}
We state here the gluing result in the anisotropic setting (for the proof cf. \cite[Proposition 11]{Ambrosio})
\begin{proposition}
    \label{gluing}
    Let $s\in(0,1)$ and $E_1$,$E_2\sub\R^n$ be measurable sets such that $\pks^1(E_j,\O)<\infty$ for $j=1,2$. Then, given $\delta_1>\delta_2>0$ there exists a measurable set $F\sub\R^n$ such that:
    \begin{enumerate}[start=1, label=\roman*.]
        \item $\norm{\chi_F-\chi_{E_1}}_{L^1(\O)}\leq\norm{\chi_{E_2}-\chi_{E_1}}_{L^1(\O)}$,
        \item $F\cap(\O\setminus\O_{\delta_1})=E_1\cap(\O\setminus\O_{\delta_{1}})$ and $F\cap\O_{\delta_2}=E_2\cap\O_{\delta_2}$,
        \item For any $\epsilon >0$ it holds that:
        \begin{equation*}
            \begin{split}
                \pks^1(F,\O)&\leq \pks^1(E_1,\O)+\pks^1(E_2,\O_{\delta_1+\epsilon})+\frac{C(\O,c)}{\epsilon^{n+s}}\\
                &+C(\O,\delta_1,\delta_2)\left[\frac{\norm{\chi_{E_1}-\chi_{E_2}}_{L^1(\O_{\delta_1}\setminus \O_{\delta_2})}}{(1-s)}+\norm{\chi_{E_1}-\chi_{E_2}}_{L^1(\O)}\right].
            \end{split}
        \end{equation*}
    \end{enumerate}
\end{proposition}

The following corollary is a direct consequence of the proposition above.
\begin{corollary}
    \label{corollaryglue}
    Suppose that $E_s\sub\R^n$, $s\in(0,1)$ and $E\sub\R^n$ are measurable sets such that $\chi_{E_s}\to\chi_E$ as $s\to1^-$ in $L^1(\O)$ and $\pks^1(E_s,\O),\pks^1(E,\O)<\infty$. Then, for any $\d_1>\d_2>0$ there exist measurable sets $F_s\sub\R^n$, $s\in(0,1)$, such that:
    \begin{enumerate}[start=1, label=\roman*.]
        \item $\chi_{F_s}\to\chi_E$ in $L^1(\O)$ as $s\to\infty$,
        \item $F_s\cap(\O\setminus\O_{\d_1})=E_s\cap(\O\setminus\O_{\d_1})$ and $F_s\cap\O_{\d_2}=E\cap\O_{\d_2}$,
        \item for any $\epsilon>0$:
        \begin{equation*}
            \liminf_{s\to1^-}(1-s)\pks^1(F_s,\O)\leq \liminf_{s\to1^-}(1-s)\pks^1(E_s,\O)+\limsup_{s\to1^-}(1-s)\pks^1(E,\O_{\d_1+\epsilon})
        \end{equation*}
    \end{enumerate}
\end{corollary}
Let us introduce some constants that will be used in the proof of the main theorem of this section.
Consider sequences $s_i\to 1^-$ and measurable sets $E_i\sub\R^n$ such that $E_i\to H$ in $L^1(Q)$. Define the following quantity:
\begin{equation}
    \label{gammandef}
    \Gamma_n\coloneq \Gamma(L^1_{loc}(\R^n))\hbox{-}\liminf_{s\to1^-}(1-s)\pks^1(H,Q).
\end{equation}

For $\delta>0$ we recall that $Q_\delta\coloneq\{x\in Q:\mathrm{dist}(x,\partial Q)<\delta\}$. Let us set:
\begin{equation}
    \Tilde{\Gamma}_n\coloneq \inf \left\{\liminf_{i\to\infty}(1-s_i)\pksi^1(F_i,Q) \right \},
\end{equation}
where the infimum is taken among all sequences of measurable sets $(F_i)_i$ such that $F_i\to H$ in $L^1(Q)$ and there exist some $\d>0$ for which, for every index $i$, we have that $F_i\cap Q_\d=H\cap Q_\d$.
We use the corollary above to prove that these two quantities are, in fact, the same. The importance of this result consists in the fact that it allows to recover the $\Gamma$-liminf \eqref{gammandef} by restricting ourselves to sequences of sets which agree with $H$ in a neighborhood of $Q^c$.
\begin{lemma}
    \label{gammaequal}
    $\Gamma_n=\Tilde{\Gamma}_n$.
\end{lemma}
\begin{proof}
   Clearly $\Gamma_n\leq\Tilde{\Gamma}_n$. Let us show that the converse also holds. Let $s_i\to1^-$ and $E_i\to H$ in $L^1(Q)$. For any $\delta\in\Q$, $\delta>0$, by corollary \ref{corollaryglue} there exist a sequence of sets $(F_i)_i$ such that $F_i\cap Q_\delta=H\cap Q_\delta$, $F_i\to H$ in $L^1(Q)$ and for every $\e\in \Q$, $\epsilon>0$:
   \begin{equation}
       \liminf_{i\to\infty}(1-s_i)\pksi^1(F_i,Q)\leq \liminf_{i\to\infty}(1-s_i)\pksi^1(E_i,Q)+\limsup_{i\to\infty}(1-s_i)\pksi^1(H,Q_{\d+\epsilon}).
   \end{equation}
   Since the kernels $k_s(\cdot)$ verify \eqref{h2}, by lemma \ref{lemma81} there exists a positive constant $C(n)<\infty$ such that:
   \begin{equation}
       \limsup_{i\to\infty}(1-s_i)\pksi^1(H,Q_{\d+\epsilon})\leq c\cdot C(n) P(H,Q_{\d+\epsilon}).
   \end{equation}
   Thus:
   \begin{equation*}
       \liminf_{i\to\infty}(1-s_i)\pksi^1(F_i,Q)\leq \liminf_{i\to\infty}(1-s_i)\pksi^1(E_i,Q)+c\cdot C(n)P(H,Q_{\d+\e}).
   \end{equation*}
   By the arbitrariness of $\e\in \Q$, $\e>0$:
   \begin{equation}
       \liminf_{i\to\infty}(1-s_i)\pksi^1(F_i,Q)\leq \liminf_{i\to\infty}(1-s_i)\pksi^1(E_i,Q)+c\cdot C(n)P(H,Q_\d),
   \end{equation}
   which implies:
   \begin{equation}
       \Tilde{\Gamma}_n\leq \liminf_{i\to\infty}(1-s_i)\pksi^1(E_i,Q)+c\cdot C(n)P(H,Q_\d).
   \end{equation}
    Letting now $\delta\to0^+$ we obtain:
    \begin{equation}
        \Tilde{\Gamma}_n\leq \liminf_{i\to\infty}(1-s_i)\pksi^1(E_i,Q),
    \end{equation}
    thus ending the proof.
\end{proof}
\par

Observe that for each $s\in(0,1)$ and each measurable set $F\sub\R^n$ we may write:
    \begin{equation}
        \label{phqeq}
        \begin{split}
            \pks^1(F,Q)&=\I_{k_s}(F\cap Q,F^c\cap Q)\\
            &=\I_{k_s}(F\cap Q,\R^n)-\I_{k_s}(F\cap Q,F^c\cap Q^c)-\I_{k_s}(F\cap Q, F\cap Q^c), 
        \end{split}
    \end{equation}
which gives the following decomposition:
\begin{equation}
    \label{decomposition}
    \I_{k_s}(F\cap Q,\R^n)=\pks^1(F,Q)+\I_{k_s}(F\cap Q,F^c\cap Q^c)+\I_{k_s}(F\cap Q, F\cap Q^c).
\end{equation}
We can now state and prove the main result of this section.
\begin{theorem}[$\Gamma$-liminf inequality]
\label{gammaliminf}
For every sequence $s_i\to1^-$, $(E_i)_i$ and $E$ measurable sets such that $\chi_{E_i}\to\chi_E$ in $L^1_{loc}(\R^n)$:
\begin{equation*}
    \liminf_{i\to\infty}(1-s_i)\pksi^1(E_i,\O)\geq P_{ZB_k}(E,\O).
\end{equation*}
\begin{proof}
    Assume that:
    \begin{equation*}
         \liminf_{i\to\infty}(1-s_i)\pksi^1(E_i,\O) < \infty,
    \end{equation*}
    otherwise the statement is obvious. Upon passing to a subsequence we can also assume that:
    \begin{equation*}
        \liminf_{i\to\infty}(1-s_i)\pksi^1(E_i,\O)=\lim_{i\to\infty}(1-s_i)\pksi^1(E_i,\O).
    \end{equation*}
    Passing to the limit with respect to $i$ in \eqref{tbound} with $\O$ instead of $\O'$, we get that for every $K\Subset \O$:
    \begin{equation*}
        \norm{\tau_h \chi_E-\chi_E}_{L^1(K)}\leq C(n,\O)\abs{h}
    \end{equation*}
    for $\abs{h}$ small enough, thus $E$ has finite perimeter in $\O$ (\cite[Proposition 9.3 and subsequent remark]{Brezis}). Let us call $\mu$ the anisotropic perimeter measure of $E$ in $\O$ with respect to $ZB_k$, that's to say the Radon measure on $\O$ defined on compact sets $K\sub \O$ by:
    \begin{equation*}
        \mu(K)\coloneq\int_{\partial^*E\cap K}\abs{\nu_E(x)}_{Z^*B_k}d\H^{n-1}(x).
    \end{equation*}
    Let $\mathcal{C}$ be the set of all open cubes contained in $\O$. Recalling that:
    \begin{equation*}
        Q=\left(-\frac{1}{2},\frac{1}{2}\right)^n
    \end{equation*}
    denotes the standard unit open cube in $\R^n$ centered in the origin, every cube $C\in\mathcal{C}$ can be written as:
    \begin{equation*}
        C=x+rRQ,
    \end{equation*}
    for some $x\in\O$, $r>0$ and $R\in SO(n)$.
    Since $E$ has finite perimeter in $\O$ for every $x\in\partial^*E\cap \O$ there exist a unique $R_x\in SO(n)$ such that:
    \begin{equation*}
        \frac{E-x}{r}\to R_xH
    \end{equation*}
    as $r\to0^+$ in $L^1_{loc}(\R^n)$ (\textit{blow-up})(\cite[Theorem 1, Section 5.7]{Evansgariepy}). Moreover, calling $C(x,r)=x+rR_xQ$ it holds that (\cite[Corollary 1, Section 5.7]{Evansgariepy}):
    \begin{equation*}
        \lim_{r\to0^+}\frac{\norm{\partial E}(C(x,r))}{r^{n-1}}=1.
    \end{equation*}
    Since $\norm{\partial E}=\H^{n-1}\resmes \partial^*E $,
    using the Lebesgue-Besicovitch differentiation theorem (\cite[Theorem 1, Section 1.7]{Evansgariepy}),  we get that for any $x\in \partial^*E$:
    \begin{equation*}
        \begin{split}
            \lim_{r\to0^+}\frac{\mu(C(x,r))}{r^{n-1}} & = \lim_{r\to 0^+} \frac{\norm{\partial E}(C(x,r))}{r^{n-1}}\frac{1}{\norm{\partial E}(C(x,r))}\int_{C(x,r)}\abs{\nu_E}_{Z^*B_k}d \norm{\partial E}\\
            & =\abs{\nu_E(x)}_{Z^*B_k},
        \end{split}
    \end{equation*}
    which we can rewrite as:
    \begin{equation}
        \label{mulim}
        \lim_{r\to0^+}\frac{\mu(C(x,r))}{r^{n-1}\abs{\nu_E(x)}_{Z^*B_k}}=1.
    \end{equation}
    For a cube $C\in\mathcal{C}$ define:
    \begin{equation*}
        \a_i(C)=(1-s_i)\pksi^1(E_i,C)
    \end{equation*}
    and
    \begin{equation*}
        \a(C)=\liminf_{i\to\infty}\a_i(C).
    \end{equation*}
    We claim that:
    \begin{equation}
        \label{maininequality}
        \liminf_{r\to0^+}\frac{\a(C(x,r))}{\mu(C(x,r))}\geq 1.
    \end{equation}
    To prove so, fix $x\in\partial^*E\cap\O$. Without loss of generality we can assume that $x=0$ and $R_x=I$. By \eqref{mulim}:
    \begin{equation*}
        \liminf_{r\to0^+}\frac{\a(C(x,r))}{\mu(C(x,r))}=\liminf_{r\to0^+}\frac{\a(C(x,r))}{r^{n-1}\abs{e_n}_{Z^*B_k}}=\liminf_{r\to0^+}\frac{\a(rQ)}{r^{n-1}\abs{e_n}_{Z^*B_k}}.
    \end{equation*}
    Let $r_j\to0^+$ be a sequence such that:
    \begin{equation*}
        \liminf_{r\to0^+}\frac{\a(rQ)}{r^{n-1}\abs{e_n}_{Z^*B_k}}=\lim_{j\to\infty}\frac{\a(r_j Q)}{r_j^{n-1}\abs{e_n}_{Z^*B_k}}.
    \end{equation*}
    Let us choose a sequence of integers $i(j)\to\infty$ such that for every $j\geq 1$:
    \begin{enumerate}[start=1, label=\roman*.]
        \item $r_j^{1-s_{i(j)}}\geq\left(1-\frac{1}{j}\right)$,
        \item $\a(r_jQ)\geq\a_{i(j)}(r_jQ)-r_j^n$,
        \item \begin{equation*}
                \int_Q\abs{\chi_{E_{i(j)/r_j}}-\chi_H}dx<\frac{1}{j}. 
            \end{equation*}
    \end{enumerate}
    With this choice of $i(j)$ we get:
    \begin{equation*}
        \begin{split}
            \frac{\a(r_jQ)}{r_j^{n-1}\abs{e_n}_{Z^*B_k}} &\geq \frac{\a_{i(j)}(rQ)}{r_j^{n-1}\abs{e_n}_{Z^*B_k}}-\frac{r_j}{\abs{e_n}_{Z^*B_k}}\\
            &=\frac{(1-s_{i(j)})P^1_{k_{s_{i(j)}}}(E_{i(j)}/r_j,Q)r_j^{(n-s_{i(j)})}}{r_j^{n-1}\abs{e_n}_{Z^*B_k}}-\frac{r_j}{\abs{e_n}_{Z^*B_k}}\\
            &\geq \left(1-\frac{1}{j}\right)\frac{(1-s_{i(j)})P^1_{k_{s_{i(j)}}}(E_{i(j)}/r_j,Q)}{\abs{e_n}_{Z^*B_k}}-\frac{r_j}{\abs{e_n}_{Z^*B_k}}.
        \end{split}
    \end{equation*}
    Then:
    \begin{equation*}
        \begin{split}
            \lim_{j\to\infty}\frac{\a(r_jQ)}{r_j^{n-1}\abs{e_n}_{Z^*B_k}} & \geq\frac{1}{\abs{e_n}_{Z^*B_k}}\liminf_{j\to\infty}(1-s_{i(j)})P^1_{k_{s_{i(j)}}}(E_{i(j)}/r_j,Q)\\
            & \geq \frac{1}{\abs{e_n}_{Z^*B_k}}\inf\left\{ \liminf_{i\to\infty}(1-s_i)\pksi^1(G_i,Q):G_i\to H\textit{ in }L^1(Q)\right\}.\\
            &=\frac{1}{\abs{e_n}_{Z^*B_k}} \Gamma_n.
        \end{split}
    \end{equation*}
    Thus, by lemma \ref{gammaequal}:
    \begin{equation}
        \lim_{j\to\infty}\frac{\a(r_jQ)}{r_j^{n-1}\abs{e_n}_{Z^*B_k}}\geq \frac{1}{\abs{e_n}_{Z^*B_k}} \Tilde{\Gamma}_n.
    \end{equation}
    To get \eqref{maininequality}, we have to show that:
    \begin{equation}
        \label{tildegammain}
        \Tilde{\Gamma}_n\geq \abs{e_n}_{Z^*B_k}.
    \end{equation}
    To do so, let $s_i\to1^-$ and let $F_i$ be a sequence of measurable sets as in the definition of $\Tilde{\Gamma}_n$. Since we only want to estimate $(1-s_i)\pksi^1(F_i,Q)$ we can assume that $F_i\cap Q^c=H\cap Q^c$. By exploiting decomposition \eqref{decomposition} we can write:
    \begin{equation}
        \begin{split}
            (1-s_i)\pksi(F_i\cap Q,\R^n)=&(1-s_i)\pksi^1(F_i,Q)\\
            &+(1-s_i)\int_{F_i\cap Q}\int_{F_i^c\cap Q^c}k_{s_i}(x-y)dxdy\\
            &+(1-s_i)\int_{F_i\cap Q}\int_{F_i\cap Q^c}k_{s_i}(x-y)dxdy\\
            &\coloneq I+II+III.
        \end{split}
    \end{equation}
    By theorem \ref{gammawholespace}:
    \begin{equation*}
        \liminf_{i\to\infty}(1-s_i)\pksi(F_i\cap Q,\R^n)\geq P_{ZB_k}(H\cap Q,\R^n).
    \end{equation*}
    As for $II$, since $F_i\Delta H\sub Q\setminus Q_\d$:
    \begin{equation*}
        (1-s_i)\int_{F_i\cap Q}\int_{F_i^c\cap Q^c}k_{s_i}(x-y)dxdy=(1-s_i)\int_{H\cap Q}\int_{H^c\cap Q^c}k_{s_i}(x-y)dxdy+\o_{s_i},
    \end{equation*}
    where $\o_{s_i}\to0$ as $s_i\to1^-$.
    By lemma \ref{lemma82}:
    \begin{equation*}
        \lim_{i\to\infty}(1-s_i)\int_{H\cap Q}\int_{H^c\cap Q^c}k_{s_i}(x-y)dxdy=0,
    \end{equation*}
    thus:
    \begin{equation*}
        \lim_{i\to\infty}(1-s_i)\int_{F_i\cap Q}\int_{F_i^c\cap Q^c}k_{s_i}(x-y)dxdy=0.
    \end{equation*}
    Let us now compute the limit as $i\to\infty$ of $III$. Since the kernels $k_s(\cdot)$ are origin-symmetric we can write:
    \begin{equation}
        \begin{split}
             (1-s_i)\int_Q\int_{Q^c}k_{s_i}(x-y)dxdy=&2(1-s_i)\int_{H\cap Q}\int_{H\cap Q^c}k_{s_i}(x-y)dxdy\\
             &+2(1-s_i)\int_{H\cap Q}\int_{H^c\cap Q^c}k_{s_i}(x-y)dxdy.
        \end{split}
    \end{equation}
    By lemma \ref{lemma82}, we get that the last summand goes to $0$ as $s_i\to1$. Since by theorem \ref{pointwholespace}:
    \begin{equation}
     \lim_{i\to\infty}(1-s_i)\int_Q\int_{Q^c}k_{s_i}(x-y)dxdy=P_{ZB_k}(Q,\R^n),   
    \end{equation}
    we finally get that:
    \begin{equation}
        \begin{split}
            \lim_{i\to\infty}(1-s_i)\int_{F_i\cap Q}\int_{F_i\cap Q^c}k_{s_i}(x-y)dxdy&=\lim_{i\to\infty}(1-s_i)\int_{H\cap Q}\int_{H\cap Q^c}k_{s_i}(x-y)dxdy\\
            &=\frac{1}{2}P_{ZB_k}(Q,\R^n).
        \end{split}
    \end{equation}
    Putting all together we obtain:
    \begin{equation}
        \begin{split}
            \liminf_{i\to\infty} (1-s_i)\pksi^1(F_i,Q)& \geq  P_{ZB_k}(H\cap Q,\R^n)-\frac{1}{2}P_{ZB_k}(Q,\R^n)\\
            &=\abs{e_n}_{Z^*B_k},
        \end{split}
    \end{equation}
    and \eqref{tildegammain} is proved.\\
    To conclude the proof, for any $\epsilon>0$ define:
    \begin{equation}
        \mathcal{A}_\epsilon \coloneq \left\{C(x,r)\sub\O:\mu(C(x,r))\leq(1+\epsilon)\alpha(C(x,r))\right\}.
    \end{equation}
    By the previous step, $\mathcal{A}_\epsilon$ is a fine covering of $\mu$-almost every $x\in\O$, meaning that for $\mu$-almost every $x\in \O$ it holds that $\inf\{r>0:C(x,r)\in\mathcal{A}_\epsilon\}=0$. Since $E$ has finite perimeter in $\O$, $\mu(\O)<\infty$ and by the Besicovitch covering theorem (\cite[Section 1.5.2]{Evansgariepy}) there exists a countable subset of disjoint cubes:
    \begin{equation*}
        \left\{C_j\sub\O:j\in\N\right\}\sub \mathcal{A}_\epsilon,
    \end{equation*}
    such that:
    \begin{equation*}
        \mu\left(\O\setminus\bigcup_j C_j\right)=0.
    \end{equation*}
    Thus, we may estimate:
    \begin{equation*}
        \begin{split}
            P_{ZB_k}(E,\O)=\mu(E)=\mu\bigg(\bigcup_j C_j\bigg)=\sum_j\mu(C_j) & \leq (1+\epsilon)\sum_j\a(C_j)\\
            &\leq (1+\epsilon) \sum_j\liminf_{i\to\infty}(1-s_i)\pksi^1(E_i,C_j)\\
            &\leq (1+\epsilon) \liminf_{i\to\infty}\sum_j(1-s_i)\pksi^1(E_i,C_j)\\
        \end{split}
    \end{equation*}
    Since the functional $\pksi^1(E_i,\cdot)$ is super-additive on disjoint sets we get that:
    \begin{equation*}
        P_{ZB_k}(E,\O)\leq (1+\epsilon)\liminf_{i\to\infty}(1-s_i)\pksi^1(E_i,\O).
    \end{equation*}
    By the arbitrariness of $\epsilon>0$ the proof is complete.    
\end{proof}
\end{theorem}

\subsection{Proof of the $\Gamma$-liminf Inequality in the One-dimensional Case}
\label{sectiondim1}
We reserve this section to the proof of Proposition \ref{1dliminf}. The structure of the proof is the same as the general case but the computations appear to be simplified. We recall that we have to prove that, if $s_i\to1^-$ and $E,E_i\sub\R$ are bounded measurable sets such that $E_i\to E$ in $L^1_{loc}(\R)$, then:
\begin{equation}
    \liminf_{i\to\infty}(1-s_i)P_{s_i}(E_i,\R)\geq P(E,\R).
\end{equation}
Suppose that:
\begin{equation*}
     \liminf_{i\to\infty}(1-s_i)P_{s_i}(E_i,\R)<\infty,
\end{equation*}
otherwise the statement is obvious. Without loss of generality let us also assume that:
\begin{equation*}
     \liminf_{i\to\infty}(1-s_i)P_{s_i}(E_i,\R)= \lim_{i\to\infty}(1-s_i)P_{s_i}(E_i,\R).
\end{equation*}
By the equi-coercivity result (Theorem \ref{equicoerc}) $E$ has finite perimeter, so we may assume that $E$ is the union of a finite number of intervals lying at mutually positive distance.\\
Fix $x_0\in\partial^*E$. For $r>0$ small enough we have that $U(x_0,r)\coloneq(x_0-r/2,x_0+r/2)$ intersects $\partial E$ only in $x_0$, which means that $\norm{\partial E}(U(x_0,r))=1$. Let us assume that $x_0$ is the right endpoint of one of the intervals of $E$.
Following the steps of the general case we call
\begin{equation}
    \a_i(U(x_0,r))=(1-s_i)P^1_{s_i}(E_i,U(x_0,r)),
\end{equation}
and
\begin{equation}
    \a(U(x_0,r))=\liminf_{i\to\infty}\a_i(U(x_0,r)).
\end{equation}
We claim that:
\begin{equation*}
    \label{limalpha1}
    \liminf_{r\to 0^+}\a(U(x_0,r))\geq 1.
\end{equation*}
Let $r_j\to 0^+$ be a sequence such  that:
\begin{equation*}
    \liminf_{r\to 0^+}\a(U(x_0,r))=\lim_{j\to\infty}\a(U(x_0,r_j)).
\end{equation*}
Foer each $j$ let us choose an index $i(j)$ such that:
\begin{enumerate}[label=\roman*.]
    \item $\a(U(x_0,r_j))\geq \a_{i(j)}(U(x_0,r_j))-\frac{1}{j}$
    \item $r_j^{1-s_{i(j)}}\geq \left(1-\frac{1}{j}\right)$
    \item \begin{equation*}
        \frac{1}{r_j}\int_{U(x_0,r_j)}\abs{\chi_{E_{i(j)}}-\chi_E}dx\leq\frac{1}{j}.
    \end{equation*}
\end{enumerate}
Then we may estimate:
\begin{equation*}
    \begin{split}
        \a(U(x_0,r_j))&\geq  \a_{i(j)}(U(x_0,r_j))-\frac{1}{j}\\
        &=(1-s_{i(j)})P^1_{s_{i(j)}}\left( \frac{E_{i(j)-x_0}}{r_j},U(0,1)\right)\,r_j^{1-s_{i(j)}}-\frac{1}{j}\\
        &\geq \left(1-\frac{1}{j}\right)(1-s_{i(j)})P^1_{s_{i(j)}}\left( \frac{E_{i(j)-x_0}}{r_j},U(0,1)\right)-\frac{1}{j}.
    \end{split}
\end{equation*}
It follows that:
\begin{equation*}
    \lim_{j\to\infty}\a(U(x_0,r_j))\geq \Gamma_1,
\end{equation*}
where $\Gamma_1$ is defined in \eqref{gammandef}. By Lemma \ref{gammaequal}, we only need to prove that, for every $\d>0$ and for every sequence of measurable sets $F_i\sub\R$ such that $F_i\to H_1\coloneq(-\infty,0]$ in $L^1_{loc}(\R)$ and $F_i=H_1$ on $\{x\in\R:\mathrm{dist}(x,U(0,1)^c)<\delta\}$, there holds:
\begin{equation}
    \liminf_{i\to\infty}(1-s_i)P^1_{s_i}(F_i,U(0,1))\geq 1.
\end{equation}
By the result on the minimality of halfspaces (\cite[Proposition 17]{Ambrosio}) we may estimate:
\begin{equation*}
    \begin{split}
        (1-s_i)P^1_{s_i}(F_i,U(0,1))&=(1-s_i)P_{s_i}(F_i,U(0,1))-(1-s_i)P^2_{s_i}(F_i,U(0,1))\\
        &\geq (1-s_i)P_{s_i}(H_1,U(0,1))-(1-s_i)P^2_{s_i}(F_i,U(0,1)).
    \end{split}
\end{equation*}
By construction:
\begin{equation*}
    \lim_{i\to\infty}(1-s_i)P^2_{s_i}(F_i,U(0,1))=0.
\end{equation*}
Thus, by an easy computation we obtain:
\begin{equation*}
    \liminf_{i\to\infty}(1-s_i)P^1_{s_i}(F_i,U(0,1))\geq \lim_{i\to\infty}(1-s_i)P_{s_i}(H_1,U(0,1))=1.
\end{equation*}
Let us now write $\partial^* E=\{x_1,...,x_m\}$ and let us fix $\e>0$. It follows from the claim above that there exist $r_1,...,r_m>0$ small enough such that $U(x_j,r_j)\cap U(x_h,r_h)=\emptyset$ for $i\neq j$, $\norm{\partial E}(U(x_j,r_j))=1$ for each $j$ and:
\begin{equation*}
    \a(U(r_j,x_j))\geq\frac{1}{1+\e}.
\end{equation*}
Hence we get the following estimate:
\begin{equation*}
    \begin{split}
        P(E,\R)=\H^0(\partial^*E)&\leq (1+\e)\sum_{j=1}^m \a(U(x_j,r_j))\\
        &\leq (1+\e) \liminf_{i\to\infty}(1-s_i)\sum_{j=1}^m P^1_{s_i}(E_i,U(x_j,r_j))\\
        &\leq (1+\e) \liminf_{i\to\infty}(1-s_i)P^1_{s_i}(E_i,\R).
    \end{split}
\end{equation*}
Since this holds for every $\e>0$ the proof is complete.

\subsection{$\Gamma$-limsup Inequality}
\label{proofgammalimsupbounded}

This section is dedicated to proving the $\Gamma(L^1_{loc})$-limsup inequality \eqref{limsup}.  We have to prove that for every measurable set $E\sub\R^n$ with finite perimeter in $\O$ and for every sequence $s_i\to1^-$ there exists a sequence of measurable sets $(E_i)_i$ such that $E_i\to E$ in $L^1_{loc}(\R^n)$ and
\begin{equation}
    \label{limsup2}
    \limsup_{i\to\infty}(1-s_i)\pksi(E_i,\O)\leq \pzbk(E,\O).
\end{equation}
From theorem \ref{gammaliminf} it follows that for such a sequence we have that:
\begin{equation*}
    \pzbk(E,\O)=\lim_{i\to\infty}(1-s_i)\pksi^1(E_i,\O)=\lim_{i\to\infty}(1-s_i)\pksi(E_i,\O),
\end{equation*}
and $(1-s_i)\pksi^2(E_i,\O)\to0$ as $s_i\to1^-$.\\
We will prove the inequality for a collection of subsets which is dense in energy with respect to the anisotropic perimeter $P_{ZB_k}(\cdot,\O)$. We say that a collection $\mathcal{B}$ of subsets of $\R^n$ is dense in energy if for every measurable set with finite perimeter $E\sub\R^n$ there exist a sequence $(E_j)_j\sub\mathcal{B}$ such that $E_j\to E$ in $L^1_{loc}(\R^n)$ and:
\begin{equation}
    \label{density}
    \limsup_{j\to\infty}P_{ZB_k}(E_j,\O)=P_{ZB_k}(E,\O).
\end{equation}
If the limsup inequality is proved for such a collection $\mathcal{B}$, then it holds also for ever measurable set $E\sub\R^n$ with finite perimeter. In fact, consider an arbitrary sequence  $s_i\to1^-$ and let  $(E_j)_j\sub\mathcal{B}$ be as above. Let $d(\cdot,\cdot)$ be a distance inducing the $L^1_{loc}(\R^n)$ convergence. If \eqref{limsup2} holds for sets in $\mathcal{B}$, for every index $j$ we may find a measurable set $\hat{E}_j$ and an index $i(j)>i(j-1)$ such that $d(E_j,\hat{E}_j)<1/j$ and:
\begin{equation*}
    (1-s_{i(j)})P_{k_{s_{i(j)}}}(\hat{E}_j,\O)\leq \pzbk(E_j,\O)+\frac{1}{j}.
\end{equation*}
Then $\hat{E_j}\to E$ in $L^1_{loc}(\R^n)$ and:
\begin{equation*}
    \limsup_{j\to\infty}(1-s_{i(j)})P_{k_{s_{i(j)}}}(\hat{E}_j,\O)\leq \pzbk(E,\O).
\end{equation*}
As for the choice of $\mathcal{B}$, we may choose the collection of measurable sets $\Pi\sub\R^n$ with polyhedral boundary such that:
\begin{equation}
    \label{partialper}
    P(\Pi,\partial\O)=0.
\end{equation}
This is possible due to the following approximation result for sets with finite perimeter in $\O$ (\cite[Theorem~A.4.]{Comi}, \cite[Proposition 15]{Ambrosio}).
\begin{theorem}
    Let $\O\sub\R^n$ be a bounded open set with Lipschitz boundary and let $E\sub\R^n$ be a measurable set such that $P(E,\O)<\infty$. Then there exists a sequence of measurable sets with polyhedral boundary $(\Pi_j)_j$ such that for every $j$:
    \begin{equation*}
        P(\Pi_j,\partial\O)=0,
    \end{equation*}
    $\Pi_j\to E$ in $L^1_{loc}(\R^n)$ and $P(\Pi_j,\O)\to P(E,\O)$ as $j\to\infty$.
\end{theorem}
In fact, by exploiting Reshetnyak's continuity theorem (\cite[Theorem 20.12]{Maggi}), we may also deduce that for the sequence $(\Pi_j)_j$ provided by the theorem above it holds that:
\begin{equation*}
    \pzbk(\Pi_j,\O)\to\pzbk(E,\O),
\end{equation*}
as $j\to\infty$.\\
Note that property \eqref{partialper} can be rephrased as follows. We recall that for any $\d>0$ we had set in the previous section:
\begin{equation*}
    \O_\d\coloneq\{x\in\O:\mathrm{dist}(x,\O^c)<\d\},
\end{equation*}
and:
\begin{equation*}
    \O_\d^c\coloneq\{x\in\O^c:\mathrm{dist}(x,\O)<\d\}.
\end{equation*}
Then \eqref{partialper} holds if and only if:
\begin{equation*}
    \lim_{\d\to0^+}P(\P,\O_\d\cup\O_\d^c)=0.
\end{equation*}
We proceed now to state and prove the main result of this section, which is the adaptation of \cite[Lemma 8]{Ambrosio} to the anisotropic setting.

\begin{theorem}
    \label{gammalimsup}
    Let $\Pi\sub\R^n$ be a measurable set with polyhedral boundary such that $P(\Pi,\partial\O)=0$. Then:
    \begin{equation}
        \limsup_{s\to1^-}(1-s)\pks(\Pi,\O)\leq \pzbk(\Pi,\O).
    \end{equation}
\end{theorem}
\begin{proof}
    We recall that:
    \begin{equation*}
         \pks(\P,\O)=\pks^1(\P,\O)+\pks^2(\P,\O).
    \end{equation*}
    The proof is divided into two steps. First, we prove that:
    \begin{equation}
        \limsup_{s\to1^-}(1-s)\pks^1(\P,\O)\leq\pzbk(\Pi,\O).
    \end{equation}
    Then we move on to showing that:
    \begin{equation}
        \limsup_{s\to1^-}(1-s)\pks^2(\P,\O)=0.
    \end{equation}
    \textit{Step 1.} For $\e>0$ we call:
    \begin{equation*}
        \partial_\e \P\coloneq\{x\in \O:d(x,\partial\P) <\e/2 \},
    \end{equation*}
    and:
    \begin{equation*}
        \partial_\e\P^-\coloneq \partial_\e\P\cap\P.
    \end{equation*}
    For each $\e>0$ consider finite collections of $N_\e$ cubes $(Q_i^\e)_i$, $i=1,...,N_\e$, such that:
    \begin{enumerate}[start=1, label=\roman*.]
        \item $Q_i^\e$ has side length $\e$, $Q_i^\e\sub \O$ and its barycenter belongs to $\partial \P\cap \O$ for each $i=1,...,N_\e$. 
        \item If we denote by $\Tilde{Q}_i^\e$ the dilation of $Q_i^\e$ by a factor $(1+\e)$, then each $\Tilde{Q}_i^\e$ has non empty intersection with only one face $\Sigma$ of the polyhedral boundary of $\P$ and its faces are either parallel or orthogonal to $\Sigma$,
        \item Calling $P_\e\coloneq \bigcup_{i=1}^{N_\e}Q_i^\e$, it holds that $\H^{n-1}((\partial\P\cap\O)\setminus P_\e)\to 0$ as $\e\to0^+$.
    \end{enumerate}
    For every $x\in\P\cap\O$ set:
    \begin{equation}
        I_s(x)\coloneq \int_{\P^c\cap\O}k_s(x-y)dy.
    \end{equation}
    With this notation we may write:
    \begin{equation*}
        \begin{split}
            \pks^1(\P,\O)=\int_{\P\cap\O}I_s(x)dx&=\int_{(\P\cap\O)\setminus\dpe^-}I_s(x)dx
            +\int_{\dpe^-\setminus P_\e}I_s(x)dx
            +\int_{P_\e\cap \P}I_s(x)dx\\
            &=I+II+III.
        \end{split}
    \end{equation*}
    Let us estimate each of the last three integrals separately.
    \par
    $I.$ If $x\in (\P\cap\O)\setminus\dpe^-$, then for every $y\in \P^c\cap\O$ we have that $\abs{x-y}\geq\e/2$; thus, since  the kernels $k_s(\cdot)$ verify hypothesis \eqref{h2} :
    \begin{equation}
        I_s(x)\leq\int_{B(x,\e/2)^c}k_s(x-y)dy\leq c\int_{B(x,\e/2)^c}\frac{1}{\abs{x-y}^{n+s}}dy=c\frac{n\o_n 2^s}{s\e^s},
    \end{equation}
    which gives:
    \begin{equation}
        \int_{(\P\cap\O)\setminus\dpe^-}I_s(x)dx\leq c\frac{n\o_n 2^s \L^n(\P\cap\O)}{s\e^s}.
    \end{equation}
    Hence, for every $\e>0$:
    \begin{equation}
        \limsup_{s\to1^-}(1-s) \int_{(\P\cap\O)\setminus\dpe^-}I_s(x)dx=0.
    \end{equation}
    \par
    $II.$ Let us fix $x\in \dpe^-\setminus P_\e$ and set $\d_x=d(x,\P^c\cap\O)$. Then, following the same argument as above we may estimate:
    \begin{equation}
        I_s(x)\leq \int_{B(x,\d_x)^c}k_s(x-y)dy\leq c\frac{n\o_n}{s\d_x^s}.
    \end{equation}
    Let us now write:
    \begin{equation*}
        \partial\P\cap\O=\bigcup_{j=1}^M \Sigma_j,
    \end{equation*}
    where every $\Sigma_j$ is the intersection of a face of the boundary of $\P$ with $\O$. Moreover, for every $j=1,...,m$ define $\Sigma_{j,\e}$ as the sets of all points belonging to the same hyperplane as $\Sigma_j$ whose distance from $\Sigma_j$ is at most $\e$, and
    \begin{equation*}
        \dpej^-\coloneq \{x\in \dpe^-:\d_x=d(x,\Sigma_j)\}.
    \end{equation*}
    We have that for every $j=1,..,m$, calling $\nu_j$ the inward pointing normal unit vector to $\Sigma_{j,\e}$:
    \begin{equation*}
        \dpej^-\sub \{x+t\nu:x\in \Sigma_{j,\e}\textit{ and }t\in(0,\e/2)\},
    \end{equation*}
    and also:
    \begin{equation*}
        \dpe^-=\bigcup_{j=1}^m\dpej^-.
    \end{equation*}
    Thus, we may estimate:
    \begin{equation*}
        \begin{split}
            \int_{\dpe^-\setminus P_\e}I_s(x)dx&\leq c\frac{n\o_n}{s}\sum_{j=1}^m\int_{\dpej^-\setminus P_\e}\frac{1}{\d_x^s}dx\\
            &\leq c\frac{n\o_n}{s}\sum_{j=1}^m\int_{\dpej^-\setminus P_\e}\frac{1}{d(x,\Sigma_{j,\e})}dx\\
            &\leq c\frac{n\o_n}{s} \sum_{j=1}^m \int_{\Sigma_{j,\e}\setminus P_\e}\left(\int_0^{\e/2}\frac{1}{t^s}dt\right)d\H^{n-1}(x)\\
            &=c\frac{n\o_n\e^{1-s}}{2^{1-s}s(1-s)}\H^{n-1}\left(\bigcup_j \Sigma_{j,\e}\setminus P_\e\right).
        \end{split}
    \end{equation*}
    Since as $\e\to0^+$
    \begin{equation*}
        \H^{n-1}\left(\bigcup_j \Sigma_{j,\e}\setminus P_\e\right)\to 0
    \end{equation*}
    independently of $s\in(0,1)$, we finally have that:
    \begin{equation}
        \limsup_{s\to1^-}(1-s)\int_{\dpe^-\setminus P_\e}I_s(x)dx\leq \o_\e,
    \end{equation}
    where $\o_\e\to 0$ as $\e\to 0^+$.\\  
    $III.$ For $x\in \P\cap P_\e$, write:
    \begin{equation*}
        \begin{split}
             I_s(x)&=\int_{(\P^c\cap\O)\cap B(x,\e^2/2)^c}k_s(x-y)dy+\int_{(\P^c\cap\O)\cap B(x,\e^2/2)}k_s(x-y)dy\\
             &=I^1_s(x)+I^2_s(x).
        \end{split}
    \end{equation*}
    As for $I^1_s(x)$, with the same argument as in the previous cases we get:
    \begin{equation*}
        I^1_s(x)\leq c\frac{n\o_n2^s}{s\e^{2s}},
    \end{equation*}
    which gives, since every $Q_i^\e$ is contained in $\O$:
    \begin{equation}
        \int_{\P\cap P_\e}I_s^1(x)dx\leq  c\frac{n\o_n 2^s \L^n(\O)}{s\e^{2s}},
    \end{equation}
    and so:
    \begin{equation}
        \limsup_{s\to1^-}(1-s) \int_{\P\cap P_\e}I_s^1(x)dx=0.
    \end{equation}
    In order to get  an estimate of the integral of $I^2_s(x)$, note that if $x\in Q_i^\e$, then, since every $\Tilde{Q}_i^{\e}$ has side length $\e(1+\e)$,  $B(x,\e^2/2)\sub\Tilde{Q}_i^\e$. So:
    \begin{equation}
        \int_{\P\cap P_\e}I^2_s(x)dx\leq \sum_{i=1}^{N_\e}\int_{\P\cap Q_i^\e}\int_{\P^c\cap\Tilde{Q}_i^\e}k_s(x-y)dxdy\leq \sum_{i=1}^{N_\e}\int_{\P\cap \Tilde{Q}_i^\e}\int_{\P^c\cap\Tilde{Q}_i^\e}k_s(x-y)dxdy.
    \end{equation}
    For every $i=1,...,N_\e$ let $\Sigma_i$ be the only face of $\partial\P\cap\O$ with which $\tildeqi$ has nonempty intersection and let $z_i\in\Sigma_i$ be the barycenter of $\tildeqi$. Moreover, call $n_i$ the outer normal unit vector to $\Sigma_i$ and $H_i$ the halfspace:
    \begin{equation*}
        H_i\coloneq \{y\in\R^n:(y-z_i)\cdot n_i\leq 0\}.
    \end{equation*}
    We this notation we may write:
    \begin{equation*}
        \begin{split}
            \int_{\P\cap P_\e}I^2_s(x)dx&\leq\sum_{i=1}^{N_\e}\int_{\P\cap \Tilde{Q}_i^\e}\int_{\P^c\cap\Tilde{Q}_i^\e}k_s(x-y)dxdy\\
            &=\sum_{i=1}^{N_\e}\pks^1(H_i,\tildeqi),
        \end{split}
    \end{equation*}
    and so:
    \begin{equation}
        \label{i2sest}
        (1-s)\int_{\P\cap P_\e}I^2_s(x)dx\leq \sum_{i=1}^{N_\e}(1-s)\pks^1(H_i,\tildeqi).
    \end{equation}
    We need to prove now that for each $i=1,...,N_\e$ it holds that:
    \begin{equation}
        \label{limqi}
        \limsup_{s\to1^-}(1-s)\pks^1(H_i,\tildeqi)=\pzbk(H_i,\tildeqi).
    \end{equation}
    Without loss of generality, we may restrict ourselves to proving:
    \begin{equation}
        \label{limq}
        \limsup_{s\to1^-}(1-s)\pks^1(H,Q)=\pzbk(H,Q),
    \end{equation}
    where $Q$ is the origin-centered open unit cube and and $H=\{x_n\leq 0\}$ as above. Actually, we will get a stronger statement, proving that the one appearing in \eqref{limq} is, in fact, a limit.\\
   By decomposition \eqref{decomposition} we have that:
    \begin{equation}
        \label{limsupdecomposition}
        \begin{split}
            \pks^1(H,Q)=\I_{k_s}(H\cap Q,\R^n)-\I_{k_s}(H\cap Q,H^c\cap Q^c)-\I_{k_s}(H\cap Q, H\cap Q^c). 
        \end{split}
    \end{equation}
    Also:
    \begin{equation*}
       \begin{split}
            \I_{k_s}(H\cap Q,H\cap Q^c)=&\I_{k_s}(Q,\R^n)-\I_{k_s}(H^c\cap Q,H^c\cap Q^c)-I_s(H^c\cap Q,H\cap Q^c)\\
            &-I_s(H\cap Q,H^c\cap Q^c).
       \end{split}
    \end{equation*}
    Since the kernels $k_s(\cdot)$ are origin-symmetric:
    \begin{equation*}
        \I_{k_s}(H\cap Q,H\cap Q^c)=\I_{k_s}(H^c\cap Q,H^c\cap Q^c).
    \end{equation*}
    Hence we get:
    \begin{equation}
        \label{2iseq}
        2 \I_{k_s}(H\cap Q,H\cap Q^c)=\I_{k_s}(Q,\R^n)-\I_{k_s}(H^c\cap Q,H\cap Q^c)-\I_{k_s}(H\cap Q,H^c\cap Q^c).
    \end{equation}
    Taking now into account \eqref{2iseq} and multiplying each member by $(1-s)$, \eqref{limsupdecomposition} can be rewritten as:
    \begin{equation}
    \label{phqfinal}
        \begin{split}
            (1-s)\pks^1(H,Q)=&(1-s)\bigg[\pks(H\cap Q,\R^n)-\I_{k_s}(H\cap Q,H^c\cap Q^c)\\
            &-\frac{1}{2}\bigg( \pks(Q,\R^n)-\I_{k_s}(H^c\cap Q,H\cap Q^c)-\I_{k_s}(H\cap Q,H^c\cap Q^c)\bigg)\bigg]\\
        \end{split}
    \end{equation}
    By theorem \ref{pointwholespace}:
    \begin{equation}
        \lim_{s\to1^-}(1-s)\pks(H\cap Q,\R^n)=\pzbk(H\cap Q,\R^n),
    \end{equation}
    and
    \begin{equation*}
        \lim_{s\to1^-}(1-s)\pks(Q,\R^n)=\pzbk(Q,\R^n).
    \end{equation*}
    Moreover, by lemma \ref{lemma82}:
    \begin{equation*}
            \lim_{s\to1^-}(1-s)\I_{k_s}(H\cap Q,H^c\cap Q^c)=\lim_{s\to1^-}(1-s)\I_{k_s}(H^c\cap Q,H\cap Q^c)=0.
    \end{equation*}
    Thus, from \eqref{phqfinal} we get that:
    \begin{equation}
        \lim_{s\to1^-}(1-s)\pks^1(H,Q)=\pzbk(H\cap Q,\R^n))-\frac{1}{2}\pzbk(Q,\R^n)=\pzbk(H,Q).
    \end{equation}
    Going back now to \eqref{i2sest}, we have that:
    \begin{equation}
        \limsup_{s\to1^-}(1-s)\int_{\P\cap P_\e}I_s^2(x)dx\leq \sum_{i=1}^{N_\e}\pzbk(H_i,\tildeqi)=\pzbk(\Pi,\O)+ \o_\e,
    \end{equation}
    where $\o_\e\to0$ as $\e\to0^+$.
    Combining now together the estimates for $I.$, $II.$ and $III.$ we obtain:
    \begin{equation}
        \limsup_{s\to1^-}(1-s)\pks^1(\Pi,\O)\leq \pzbk(\Pi,\O)+\o_\e.
    \end{equation}
    By the arbitrariness of $\e>0$ we conclude.\\
    \textit{Step 2.} The argument we use here is the same as the one used in the proof of lemma \ref{lemma82}. We wish to prove that:
    \begin{equation*}
           \limsup_{s\to1^-}(1-s)\pks^2(\P,\O)=0.
    \end{equation*}
    First, consider the integral:
    \begin{equation*}
        \int_{\P\cap \O}\int_{\P^c\cap \O^c}k_s(x-y)dxdy.
    \end{equation*}
    By following the same steps of the proof of lemma \ref{lemma82} we obtain that:
    \begin{equation*}
        \int_{\Pi\cap\O}\int_{\P^c\cup\O^c}k_s(x-y)dxdy\leq c\frac{2n\o_n\L^n(\O)}{s\d^s}+\int_{\P\cap(\O_\d\cup\O_\d^c)}\int_{\P^c\cap(\O_\d\cup\O_\d^c)}k_s(x-y)dxdy.
    \end{equation*}
    By swapping the role of $\P$ and $\P^c$ and combining both estimates we get that:
    \begin{equation*}
        \pks^2(\Pi,\O)\leq c\frac{4n\o_n\L^n(\O)}{s\d^s}+2\pks^1(\Pi,\O_\d\cup\O_\d^c),
    \end{equation*}
    which, by the result of the previous step, yields:
    \begin{equation*}
        \limsup_{s\to1^-}(1-s)\pks^2(\P,\O)\leq 2\pzbk(\P,\O_\d\cup\O_\d^c).
    \end{equation*}
    Finally, since $P(\Pi,\partial\O)=0$ we have that:
    \begin{equation*}
        \lim_{\d\to0^+}\pzbk(\P,\O_\d\cup\O_\d^c)=0.
    \end{equation*}
\end{proof}

\section{Convergence of Local Minimizers}

We show in this section that local minimizers for the anisotropic fractional perimeters $\pks(\cdot,\O)$, $s\in(0,1)$, converge in $L^1_{loc}(\R^n)$ to local minimizers of $\pzbk(\cdot, \O)$. \\
The proof of the convergence of local minimizers in the anisotropic case is essentially same as the one for the isotropic case (\cite[Theorem 3]{Ambrosio}), only minor changes being necessary. We reproduce it here for completeness. The proof of the theorem relies on some properties of monotone set functions, which we will briefly illustrate in section \ref{monsetfunct}. For a detailed presentation of the topic we refer to \cite[Chapter 14]{DalMaso}. Section \ref{proofconvlocmin} is dedicated to the proof of Theorem \ref{convlocmin}.

\subsection{Monotone Set Functions}
\label{monsetfunct}

Let us denote with $\mathcal{P}(\O)$ the set of all subsets of $\O$, and with $\mathcal{A}(\O)$ and $\mathcal{K}(\O)$ respectively the set of all open and compact subsets of $\O$.\\
A set function $\a:\mathcal{P}(\O)\to [0,+\infty]$ is said to be \textit{monotone} if, for every $A,B\sub \mathcal{P}(\O)$ such that $A\sub B$, it holds that:
\begin{equation*}
    \a(A)\leq \a(B).
\end{equation*}
The set function $\a$ is said to be \textit{regular} if for every $A\in\mathcal{A}(\O)$ we have:
\begin{equation*}
    \a(A)\coloneq \sup\{\a(K):K\in \mathcal{K}(\O),\,K\sub A\},
\end{equation*}
and for every $F\in \mathcal{P}(\O)$:
\begin{equation*}
    \a(F)\coloneq\inf\{\a(A):A\in\mathcal{A}(\O),\,F\sub\mathcal{A}\}.
\end{equation*}
Let $\a_i:\mathcal{P}(\O)\to[0,+\infty]$ be a regular monotone set function for $i\in\N$ and let $\a:\mathcal{P}(\O)\to[0,+\infty]$ be a regular monotone set function. We say that $(a_i)_i$ \textit{weakly converges} to $\a$ if the following two conditions holds:
\begin{enumerate}[label=\roman*.]
    \item for every $A\in\mathcal{A}(\O)$:
    \begin{equation*}
        \liminf_{i\to\infty}\a_i(A)\geq \a(A),
    \end{equation*}
    \item for every $K\in \mathcal{K}(\O)$:
    \begin{equation*}
        \limsup_{i\to\infty}\a_i(K)\leq \a(K).
    \end{equation*}
\end{enumerate}
The next theorem, first proved by De Giorgi and Letta in \cite{DeGiorgiLetta}, provides a compactness result with respect to the weak convergence for sequences of equi-bounded monotone set functions. We omit the proof here, for which we refer to \cite[Theorem 21]{Ambrosio}.
\begin{theorem}[De Giorgi, Letta]
    \label{Letta}
    Let $\a_i:\mathcal{P}(\O)\to[0,+\infty]$ be a sequence of regular monotone set functions. Suppose that for every $A\in\mathcal{A}(\O)$ such that $A\Subset \O$ there holds:
    \begin{equation*}
        \limsup_{i\to\infty}\a_i(A)<\infty.
    \end{equation*}
    Then, there exist a regular monotone set function $\a:\mathcal{P}(\O)\to [0,+\infty]$ and a subsequence $(\a_{i_j})_j$ such that $(\a_{i_j})_j$ weakly converges to $\a$. Moreover, if $(\a_i)_i$ is super-additive on disjoint opens sets, the so is $\a$.
\end{theorem}

\subsection{Proof of theorem \ref{convlocmin}}
\label{proofconvlocmin}

\textit{Step 1.} We start by proving \eqref{minbound}. For $\d>0$ define $\O_\d$ as in the previous sections and for every index $i$ set $F_i\coloneq E_i\cap(\O^c\cup\O_\d)$. By the additivity of the integral with respect to the integration domain we get:
\begin{equation*}
    (1-s_i)\pksi(E_i,\O\setminus\O_\d)\leq(1-s_i)\pksi(E_i,\O)-(1-s_i)\pksi^1(E_i,\O_\d).
\end{equation*}
Since every $E_i$ is a local minimizer for $\pksi(\cdot,\O)$ and $E_i=F_i$ in $\O_\d$:
\begin{equation*}
    \begin{split}
        \limsup_{i\to\infty} (1-s_i)\pksi(E_i,\O\setminus\O_\d)&\leq \limsup_{i\to\infty}(1-s_i)\left(\pksi(E_i,\O)-\pksi^1(E_i,\O_\d)\right)\\
        &\leq \limsup_{i\to\infty}(1-s_i)\left(\pksi(F_i,\O)-\pksi^1(F_i,\O_\d)\right)\\
        &=\limsup_{i\to\infty}(1-s_i)\left(\pksi^1(F_i,\O)-\pksi^1(F_i,\O_\d)+\pksi^2(F_i,\O)\right).
    \end{split}
\end{equation*}
Since $F_i\sub(\O^c\cup\O_\d)$ a direct computation shows that:
\begin{equation*}
    \limsup_{i\to\infty}(1-s_i)\left(\pksi^1(F_i,\O)-\pksi^1(F_i,\O_\d)\right)\leq \limsup_{i\to\infty}(1-s_i)\pksi^1(\O\setminus\O_\d,\O).
\end{equation*}
Moreover, recalling property $\eqref{h2}$ of the kernels $k_s(\cdot)$:
\begin{equation*}
    \limsup_{i\to\infty}(1-s_i)\pksi^1(\O\setminus\O_\d,\O)\leq c \cdot \limsup_{i\to\infty}(1-s_i)P^1_{s_i}(\O\setminus\O_\d,\O).
\end{equation*}
As we pointed out in the previous sections (see proof of theorem \ref{equicoerc}), the following relation between fractional perimeters and fractional Gagliardo semi-norms holds:
\begin{equation*}
    2 P^1_{s_i}(\O\setminus\O_\d,\O)=\F_{s_i}(\chi_{\O\setminus\O_\d},\O).
\end{equation*}
Thus, applying the estimate provided in \cite[Proposition 16]{Ambrosio} we get:
\begin{equation*}
    \begin{split}
        \limsup_{i\to\infty}(1-s_i)\left(\pksi^1(F_i,\O)-\pksi^1(F_i,\O_\d)\right)&\leq c \cdot \limsup_{i\to\infty}(1-s_i)P^1_{s_i}(\O\setminus\O_\d,\O)\\
        &=\frac{c}{2}\limsup_{i\to\infty} (1-s_i)\F_{s_i}(\chi_{\O\setminus\O_\d},\O)\\
        &\leq \frac{cn\o_n P(\O\setminus\O_\d,\R^n)}{2}.
    \end{split}
\end{equation*}
Arguing in a similar way we also get:
\begin{equation*}
    \begin{split}
        \limsup_{i\to\infty}(1-s_i)\pksi(F_i,\O)&\leq \limsup_{i\to\infty}(1-s_i)\pksi(\O,\R^n)\\
        &\leq c\cdot \limsup_{i\to\infty}(1-s_i)P_{s_i}^2(\O,\R^n)\\
        &\leq \frac{c\o_{n-1}P(\O,\R^n)}{2}.
    \end{split}
\end{equation*}
Thus, \eqref{minbound} is proved for every $\O'\sub (\O\setminus \O_\d)$, and then, since $\d>0$ is arbitrary, for every $\O'\Subset\O$.\\
\textit{Step 2.} Let us define for every index $i$ and every $A\sub\O$ open:
\begin{equation}
    \a_i(A)\coloneq(1-s_i)\pksi^1(E_i,A).
\end{equation}
Let us extend each $\a_i(\cdot)$ to every $F\sub\O$ in the following way:
\begin{equation}
    \a_i(F)\coloneq \inf\{\a_i(A):F\sub A\sub \O,\textit{ A open.}\}
\end{equation}
Every $\a_i(\cdot)$ defined this way is a regular monotone set function on $\mathcal{P}(\O)$ and by its definition is super-additive on disjoint open sets.
By \eqref{minbound} and Theorem \ref{Letta}, up to extracting a subsequence we have that $(\a_i)_i$ weakly converges in the sense of monotone set functions to a regular monotone set function $\a$ on $\mathcal{P}(\O)$, which is also super-additive on disjoint open sets.
We will now prove that for ever ball $B(x,R)\Subset\O$, $r>0$ such that $\a(\partial B(x,R))=0$ we have that $E$ is a local minimizer for $\pzbk(\cdot, B(x,R))$ and:
\begin{equation}
    \lim_{i\to\infty}(1-s_i)\pksi(E_i,B(x,\R))=\pzbk(E,B(x,R)).
\end{equation}
Let us fix $x\in\O$ and call for simplicty $B(x,R)\coloneq B_R$. Let $F\sub\O$ be a measurable set such that $E\Delta F\Subset B_r$. Choose $0<r<R$ such that
$E\Delta F\Subset B_R$. By the limsup inequality \eqref{limsup} we may find a sequence of measurable sets $F_i\to F$ in $L^1(B_R)$ such that:
\begin{equation}
    \label{recoverybr}
    \limsup_{i\to\infty}(1-s_i)\pksi(F_i,B_R)\leq\pzbk(F,B_R).
\end{equation}
Using the gluing argument over the sets $E_i$ and $F_i$ (see Proposition \ref{gluing}), for any given $\r,t>0$ such that $r<\rho<t<R$ there exist a sequence of measurable sets $G_i\sub\R^n$ such that $G_i=E_i$ in $\R^n\setminus B_t$, $G_i=F_i$ in $B_\r$ and for all $\e>0$ it holds that:
\begin{equation}
    \label{minest1}
    \begin{split}
       \pksi^1(G_i,B_R)\leq &\pksi^1(F_i,B_R)+\pksi^1(E_i,B_R\setminus \overbar{B}_{\r-\e})+\frac{C}{\e^{n+s_i}}\\
       &+C\frac{\L^n((E_i\Delta F_i)\cap(B_t\setminus B_\r))}{(1-s_i)}+C\L^n((F_i\Delta E_i)\cap B_R).
    \end{split}
\end{equation}
By Theorem \ref{gammaliminf}, using the local minimality of $E_i$ we get:
\begin{equation}
    \label{minestmain}
   \begin{split}
       \pzbk(E,B_R)&\leq \liminf_{i\to\infty}(1-s_i)\pksi^1(E_i,B_R)\\
       &\leq \liminf_{i\to\infty}(1-s_i)\pksi(E_i,B_R)\\
       &\leq\liminf_{i\to\infty}(1-s_i)\pksi(G_i,B_R)\\
       &\leq\liminf_{i\to\infty}(1-s_i)\pksi^1(G_i,B_R)+\limsup_{i\to\infty}(1-s_i)\pksi^2(G_i,B_R).
   \end{split}
\end{equation}
We now estimate:
\begin{equation*}
    \limsup_{i\to\infty}(1-s_i)\pksi^2(G_i,B_R).
\end{equation*}
Let us write:
\begin{equation*}
    \begin{split}
        \pksi^2(G_i,B_\R)&=\int_{G_i\cap B_R}\int_{G_i^c\cap B_R^c}k_{s_i}(x-y)dxdy+\int_{G_i^c\cap B_R}\int_{G_i\cap B_R^c}k_{s_i}(x-y)dxdy\\
        &=I+II.
    \end{split}
\end{equation*}
As for $I$, keeping in mind that $E_i$ and $G_i$ agree outside $B_t$, we have that for every $R<R'<\mathrm{dist}(x,\partial\O)$:
\begin{equation*}
    \begin{split}
        I=&\int_{G_i\cap B_t}\int_{E_i^c\cap B_R^c}k_{s_i}(x-y)dxdy+\int_{E_i\cap(B_R\setminus B_t)}\int_{E_i^c\cap B_R^c}k_{s_i}(x-y)dxdy\\
        \leq & c\cdot \int_{G_i\cap B_t}\int_{E_i^c\cap B_R^c}\frac{1}{\abs{x-y}^{n+s_i}}dxdy+\int_{E_i\cap(B_R\setminus B_t)}\int_{E_i^c\cap B_R^c}k_{s_i}(x-y)dxdy\\
        \leq & c\cdot \frac{C\L^n(G_i\cap B_t)}{s_i(R-t)^{s_i}}+\int_{E_i\cap(B_R\setminus B_t)}\int_{E_i^c\cap(B_{R'}\setminus B_R)}k_{s_i}(x-y)dxdy\\
        &+ c\cdot \int_{E_i\cap(B_R\setminus B_t)}\int_{E_i^c\cap B^c_{R'}}\frac{1}{\abs{x-y}^{n+s_i}}dxdy\\
        \leq &\pksi^1(E_i,B_{R'}\setminus \overbar{B}_t)+ c\cdot \frac{C}{s_i}\left(\frac{1}{(R-t)^{s_i}}+\frac{1}{(R'-R)^{s_i}}\right)
    \end{split}.
\end{equation*}
Integral $II$ is estimated in the same way (it is enough to switch the roles of $G_i$ and $G_i^c$). Hence we obtain:
\begin{equation}
    \label{minest2}
    \limsup_{i\to\infty}(1-s_i)\pksi^2(G_i,B_R)\leq 2\limsup_{i\to\infty}\pksi^1(E_i,B_{R'}\setminus \overbar{B}_t).
\end{equation}
Using now estimates \eqref{minest1} and \eqref{minest2} we get from \eqref{minestmain} that:
\begin{equation}
    \begin{split}
        \pzbk(E,B_R)\leq & \liminf_{i\to\infty}(1-s_i)\pksi^1(F_i,B_R)+3\limsup_{i\to\infty}(1-s_i)\pksi^1(E_i,B_{R'}\setminus \overbar{B}_{\r-\e})\\
        &+C\lim_{i\to\infty}\L^n((E_i\Delta F_i)\cap(B_t\setminus B_\r)).
    \end{split}
\end{equation}
Since $E_i\to E$ and $F_i\to F$ in $L^1(B_R)$, and $E=F$ in $B_t\setminus B_\r$, we have that:
\begin{equation*}
    \lim_{i\to\infty}\L^n((E_i\Delta F_i)\cap(B_t\setminus B_\r))=0.
\end{equation*}
Moreover, since $\a(\partial B_R)=0$ passing to the limit as $R'\to R$, $\r\to R$ and $\e\to 0^+$ we obtain (see \cite[Proposition 22]{Ambrosio}):
\begin{equation*}
    \lim_{R',\r,\e}\limsup_{i\to\infty}(1-s_i)\pksi^1(E_i,B_{R'}\setminus \overbar{B}_{\r-\e})=\lim_{\d\to 0^+}\limsup_{i\to\infty}\a_i(B_{R+\d}\setminus\overbar{B}_{R_\d})=0.
\end{equation*}
Hence, it follows from \eqref{recoverybr} that:
\begin{equation}
    \label{locmin}
    \begin{split}
        \pzbk(E,B_R)& \leq \liminf_{i\to\infty}(1-s_i)\pksi^1(E_i,B_R)\\
        &\leq \liminf_{i\to\infty} (1-s_i)\pksi(E_i,B_R)\\
        &\leq \liminf_{i\to\infty}(1-s_i)\pksi^1(F_i,B_R)\\
        &\leq\limsup_{i\to\infty}(1-s_i)\pksi^1(F_i,B_R)=\pzbk(F,B_R).
        \end{split}
\end{equation}
Choosing now $F=E$, inequalities \eqref{minest1} , \eqref{minestmain} and \eqref{minest2} give that for every subsequence of indexes $(i_j)_j$:
\begin{equation*}
    \label{minconv}
    \pzbk(E,B_R)=\liminf_{j\to\infty}(1-s_{i_j})P_{k_{s_{i_j}}}^1(E_{i_j},B_R)=\liminf_{j\to\infty}(1-s_{i_j})P_{k_{s_{i_j}}}(E_{i_j},B_R).
\end{equation*}
Since we may extract a further subsequence that converges to the liminf, by Urysohn's principle for sequences we get:
\begin{equation}
    \pzbk(E,B_R)=\lim_{i\to\infty}(1-s_i)\pksi^1(E_i,B_R)=\lim_{i\to\infty}(1-s_i)\pksi(E_i,B_R).
\end{equation}
\textit{Step 3.} We first notice that we may repeat the same argument of the previous step for every $\O'\Subset\O$ with Lipschitz boundary and such that $\a(\partial \O')=0$. To do so, set for any small enough $\d>0$, $J_\d(\O')\coloneq \{x\in\O:\mathrm{dist}(x,\O')<\d\}$ and $J_{-\d}(\O')\coloneq\{x\in\O':\mathrm{dist}(x,\partial\O')>\d\}$. Then, it is enough to replace $B_{R+\d}$ with $J_\d(\O')$ and $B_{R-\d}$ with $J_{-\d}(\O')$.\\
In particular, we have that for any open set $\O'\Subset\O$ with Lipschitz boundary such that $\a(\partial \O')=0$:
\begin{equation}
     \pzbk(E,\O')=\lim_{i\to\infty}(1-s_i)\pksi^1(E_i,\O')=\lim_{i\to\infty}(1-s_i)\pksi(E_i,\O').
\end{equation}
Then, by definition of weak convergence of monotone set functions we have that:
\begin{equation}
    \label{alfaleq}
    \begin{split}
        \a(\O')&\leq \liminf_{i\to\infty}\a_i(\O')\\
        &=\liminf_{i\to\infty}(1-s_i)\pksi^1(E_i,\O)\\
        &=\lim_{i\to\infty}(1-s_i)\pksi^1(E_i,\O')\\
        &=\lim_{i\to\infty}(1-s_i)\pksi(E_i,\O')=\pzbk(E,\O').
    \end{split}
\end{equation}
Let now $A\sub\O$ be any open set and let $K\Subset A$ be a compact set. Let $V$ be an open set such that $K\Subset V\Subset A$. As a consequence of Morse-Sard lemma (\cite[Lemma 13.15, Theorem 13.8]{Maggi}) we may find an open set with smooth boundary $\Tilde{V}$ such that $K\Subset \Tilde{V}\Subset A$.
Since for $\e>0$ small enough the set of all $\d\in(-\e,\e)$ such that $\a(\partial J_\d(\Tilde{V}))>0$ is at most countable, we may assume that $\a(\partial\Tilde{V})=0$.
For such $\Tilde{V}$ it holds by \eqref{alfaleq} that:
\begin{equation*}
    \a(\Tilde{V})\leq\pzbk(E,\Tilde{V}).
\end{equation*}
Since both $\a(\cdot)$ and $\pzbk(E,\cdot)$ are regular monotone set functions, this implies that:
\begin{equation*}
    \a(A)\leq\pzbk(E,A).
\end{equation*}
As a consequence, again by regularity we have that for every subset $D\sub\O$:
\begin{equation*}
    \a(D)\leq\pzbk(E,D).
\end{equation*}
In particular for every $\O'\Subset\O$ with Lipschitz boundary such that $\pzbk(E,\partial\O)=0$ we also have that $\a(\partial \O)=0$.\\
To complete the prof it remains to show that $E$ is a local minimizer for $\pzbk(\cdot,\O).$ To do so it is enough to exploit an approximation argument analogous to the one used above. Let $F\sub\R^n$ be a measurable set such that $F\Delta E\Subset \O$. Let $\O_k\Subset \O$ be a sequence of open sets with Lipschitz boundary such that $\a(\partial \O_k)=0$ and $\O_k\Supset\{x\in\O:\mathrm{dist}(x,\partial \O)\geq 1/k\}$. For $k$ big enough it holds that:
\begin{equation*}
    \pzbk(E,\O_k)\leq\pzbk(F,\O_k),
\end{equation*}
which, by regularity, implies:
\begin{equation*}
    \pzbk(E,\O)\leq\pzbk(F,\O).
\end{equation*}
Hence, $E$ is a local minimizer for $\pzbk(\cdot,\O)$ and the proof is complete.  

\section{Acknowledgments}       
I would like to thank Matteo Focardi for his support and his precious insights and the Department of Mathematics and Information Technology "Ulisse Dini" of the University of Florence, in which a great part of this work was carried out. This research was partially funded
by the Austrian Science Fund (FWF) project 10.55776/F65. For
open-access purposes, the authors have applied a CC BY public copyright
license to any author-accepted manuscript version arising from this
submission.

\bibliographystyle{plain}
\bibliography{Results/bibliography}

\end{document}